\documentclass[pdflatex,sn-mathphys-ay]{sn-jnl}
\usepackage{graphicx}%
\usepackage{multirow}%
\usepackage{amsmath,amssymb,amsfonts}%
\usepackage{amsthm}%
\usepackage{mathrsfs}%
\usepackage[title]{appendix}%
\usepackage{xcolor}%
\usepackage{textcomp}%
\usepackage{manyfoot}%
\usepackage{booktabs}%
\usepackage{algorithm}%
\usepackage{algorithmicx}%
\usepackage{algpseudocode}%
\usepackage{listings}%

\usepackage{amsmath,amsfonts,amsthm,bm} % Math packages

\usepackage{graphicx} 
\usepackage{array}
\usepackage{textcomp}
\usepackage{stfloats}
\usepackage{mathabx}
\usepackage{url}
\usepackage{verbatim}
\usepackage{multirow}
\usepackage{multicol}
\usepackage{xcolor}
\usepackage{microtype}
\usepackage{longtable}
\usepackage{tabto}
\usepackage{tabularx}
\usepackage{amsmath}
\usepackage{float}
\usepackage{subcaption}
\usepackage{caption}
\usepackage{tikz}
\usepackage{threeparttable} 
\usepackage{todonotes}
\usepackage{placeins}

\usepackage{color}
\newcommand{\mcolor}[1]{\textcolor{red}{#1}}  % Wasamon's colored text

\newcommand{\rcolor}[1]{\textcolor{red}{#1}}

\newcommand{\bea}{\begin{eqnarray}}
\newcommand{\ena}{\end{eqnarray}}
\newcommand{\beas}{\begin{eqnarray*}}
\newcommand{\enas}{\end{eqnarray*}}
\newcommand{\beq}{\begin{equation}}
\newcommand{\enq}{\end{equation}}
\def\qed{\hfill \mbox{\rule{0.5em}{0.5em}}}
\newcommand{\bbox}{\hfill $\Box$}
\newcommand{\ignore}[1]{}

\newcommand{\nn}{\nonumber}

\newcommand{\E}{\mathbb{E}}
\newcommand{\Var}{\mathrm{Var}}

\newtheorem{theorem}{Theorem}[section]
\newtheorem{corollary}[theorem]{Corollary}

\newtheorem{proposition}[theorem]{Proposition}
\newtheorem{lemma}[theorem]{Lemma}
\newtheorem{definition}[theorem]{Definition}
\newtheorem{example}[theorem]{Example}
\newtheorem{remark}[theorem]{Remark}

\begin{document}

\title[Article Title]{An approximate zero bias transformation for random sums: Applications to sampling with outliers, auto insurance, and generative AI}


\author[1,3]{\fnm{Wasamon} \sur{Jantai}}\email{wasamon.j@chula.ac.th}
\equalcont{These authors contributed equally to this work.}

\author*[2,3]{\fnm{Nathakhun} \sur{Wiroonsri}}\email{nathakhun.wir@kmutt.ac.th}
\equalcont{These authors contributed equally to this work.}

\affil[1]{\orgdiv{Department of Mathematics
		and Computer Science}, \orgname{Chulalongkorn University}, \orgaddress{\state{Bangkok}, \country{Thailand}}}
\affil*[2]{\orgdiv{Department of Mathematics}, \orgname{King Mongkut's University of Technology Thonburi}, \orgaddress{\city{Bangkok}, \country{Thailand}}}

\affil[3]{Statistics, Probability, and Data Science with R programming (SPD$\epsilon$R) research group}


\abstract{
We develop $L^1$ bounds for the difference between a test function of a random sum and a standard normal random variable, where the summands are assumed to be independent but not necessarily identically distributed. The bounds are obtained through a new version of the approximate zero bias transformation specifically developed for random sums. Although the identical distribution assumption is relaxed, the bounds are of order $1/\sqrt{n}$, matching the order of existing bounds in the literature under the same distributional assumption on the number of summands. The main results are then applied to three real-world settings: random sums obtained from simple random sampling with outliers, total insurance claims, and generative AI response times.}




\maketitle

\section{Introduction}\label{sec:introduction}

A random sum is a sum in which the number of summands is itself a random variable. It is written in the form
\bea \label{rs}
W = \sum_{i=1}^N X_i,
\ena
where $X_i$ are any random variables, and $N$ is a non-negative integer-valued random variable. In this work, we assume $X_i$ to be independent but are not necessarily identical, and $N$ is independent of all $X_i$. The random sum is very useful in modeling real-world quantities, such as total claim amount, aggregate losses, transaction volumes, total infections over a fixed time period, and total wait time of users or customers. The number of summands is random for these situations \citep{gnedenko2020random, YonghintJantai2025}.  

Distributional approximation of a random sum has been studied for decades \citep{robbins1948asymptotic,renyi1963central,sunklodas2015normal}.
%\cite{robbins1948asymptotic,renyi1963central,sunklodas2015normal}. 
Most of them assume that $X_i$ are identical. This assumption is a limitation for real-world modeling. For instance, in auto insurance, if the number of claims reaches a certain upper bound, the company may be stricter in paying the full amount requested by the auto service provider, leading to a new random claim amount. In a digital platform, wait time will be longer if the number of concurrent users exceeds a certain amount. In this work, we plan to relax the identical assumption and develop normal approximation methods for a random sum of independent nonidentical random variables using Stein's method based on the approximate zero biasing technique \citep{NW2017}. This technique has an extended application in \cite{JW25}.  

The main results are then applied to three real-world examples: sampling from a population with outliers, insurance claims, and generative AI response time. The most well-regarded sampling technique is simple random sampling. It is known that the standardized average always converges to the standard normal distribution. However, it is sometimes known that the population of interest contains outliers, which causes the number of non-outliers to be random \citep{barnett1994outliers}. The number of insurance claims is known to be random and normally modeled by a Poisson distribution %(see e.g. Klugman et al. \citeyear{klugman2012loss}; Denuit et al. \citeyear{denuit2007actuarial}). 
\citep{klugman2012loss,denuit2007actuarial}.
The claim amounts are typically assumed to be independent and identically distributed. However, the distribution of the claim amount may change during a year due to seasonality and the current number of claims. Generative AIs such as ChatGPT, Gemini, and Claude are now among the most popular tools for everyone. Their uses range from asking silly questions to helping code very difficult tasks. The wait time for answers depends on how simple the question is and how many concurrent users there are during the period. Though this exact same question has not been studied before in a probability and statistics context, statistical models used in Generative AI are summarized in \cite{dobriban2025statistical}.  

Stein's method was first introduced by Charles Stein in his seminal paper \citep{Stein72} and has become one of the most powerful methods to prove convergence in distribution. Its main advantages are that it provides non-asymptotic bounds on its target functions, which include distance and any other types, and that it can handle several situations involving dependence. For more details about the method in general, see the text \citep{CGS11} and the introductory notes \citep{Ross11}. The use of Stein's method for the normal approximation of random sums has been studied for more than 20 years \citep{barbour2006normal,chaidee2008berry}. More recently, \cite{Dal22} obtained an $L^1$ bound of order $O(1/\sqrt{n})$ under the assumption that the $X_i$ are independent and identically distributed.

\begin{comment}
\mcolor{The concept of Stein's method, particularly for normal approximation, is based on the fact that a random variable $W$ has the standard normal distribution, denoted $\mathcal{N}(0,1)$, if and only if
	\beas
	\E W f(W) = \E f'(W)
	\enas
	for all absolutely continuous functions $f$ with $\E |f'(W)| < \infty$. This equation, along with the form of the distance in \eqref{Wasdef}, leads to the differential equation
	\bea \label{steineq}
	h(w) - N h = f'_h(w)-wf_h(w),
	\ena
	where $Nh = \E h(Z)$ with $Z \sim \mathcal{N}(0,1)$, $h \in \mathcal{H}_1$, and $f_h$ is the solution to \eqref{steineq}. Taking the supremum over all $h \in \mathcal{H}_1$   in the expectation on the left-hand side of \eqref{steineq}, with $w$ replaced by a variable $W$, yields the distance between $W$ and $Z$ in \eqref{Wasdef}. Thus, one can handle the expectation on the right-hand side using the bounded solution $f_h$ of \eqref{steineq}, thereby avoiding direct computations of the distances.}

\end{comment}
%%%%%%%%%%

We recall that the $L^1$  distance between the distributions ${\cal L}(X)$ and ${\cal L}(Y)$ of real valued random variables $X$ and $Y$ is given, respectively, by
\beas \label{Wasdef}
d_1 \big({\cal L}(X),{\cal L}(Y)\big) &=& \int_{-\infty}^\infty |P(X \le t) - P(Y \le t)| dt \nn \\
&=& \sup_{h \in \mathcal{H}_1} |\E h(X)-\E h(Y)| 
\enas 
where $\mathcal{H}_1 =\{h: |h(y)-h(x)| \le |y-x|\}$. As previously mentioned, we provide an $L^1$ bound between a random sum given in \eqref{rs} and the standard normal, where we allow $X_i$ to be nonidentical by using an approximate zero bias transformation introduced in  \cite{NW2017}. We note that, as part of our generalization, the original zero biasing technique  \citep{GR97} used in \cite{EKJ09} is not applicable in our setting. 

We first recall the definition of a zero bias distribution. For a mean zero random variable $W$ with $\Var(W) = \sigma^2 < \infty$, $W^z$ has the zero-bias distribution with respect to $W$ if 
\bea \label{zerodef}
\E Wf(W) = \sigma^2\E f'(W^z)
\ena
for all absolutely continuous functions $f$ such that $\E |Wf(W)|<\infty$. The distribution of $W^z$ is known to be unique, and $W^z=W$ when $W$ is normally distributed. \cite{GR97} also presented a construction of a zero bias coupling for a sum of independent mean zero random variables by randomly selecting a summand based on its variance and replacing the selected summand by its zero bias variable. The approximate zero bias variable for the sum of independent mean zero random variables that we introduce in the next section is a variable $W^*$ defined similarly to $W^z$ with a remainder term added to the right hand side of \eqref{zerodef}.

The remainder of this work is organized as follows. The approximate zero biasing structure for random sums is introduced in Section \ref{main} along with the approximation results. Then the applications and simulations mentioned previously are presented in Section \ref{sec:app}. Section \ref{sec:sum} is devoted to discussion and conclusion.

\section{Approximate Zero-Bias Transformation and Main Bounds}\label{main}

In this section, we introduce an approximate zero-bias transformation for random sums of independent mean-zero random variables and use it to derive bounds for normal approximation.

\citet{GR97} constructed a zero-bias coupling for a sum of independent mean-zero random variables by selecting a summand with probability proportional to its variance and replacing it with a variable having its zero-bias distribution. We adapt this construction to sums with a random number of summands.

Let $X_1,X_2,\ldots$ be independent mean-zero random variables with finite variances $\Var(X_i)=\sigma_i^2$. Let $N$ be a nonnegative integer-valued random variable independent of $\{X_i\}_{i\geq1}$, and define $W$ as in \eqref{rs}, with the empty sum interpreted as zero. Conditioning on $N$ gives
\begin{equation*}
	\sigma^2:=\Var(W)=\sum_{i=1}^{\infty}\sigma_i^2P(N\geq i).
\end{equation*}
Throughout the paper, we assume that $0<\sigma^2<\infty$. For each integer $k\geq 0$, write $\mathbb{N}_k=\{k,k+1,k+2,\ldots\}$ and $[k]=\{1,2\ldots,k\}$. We now introduce the construction.

\begin{flushleft}
	\textbf{Construction of an approximate zero-bias variable}
\end{flushleft}

\begin{enumerate}
	\item For each $i\in\mathbb{N}$, let $X_i^*$ have the zero-bias distribution of $X_i$. Choose these variables to be mutually independent and independent of $N$ and $\{X_i\}_{i\geq1}$.
	
	\item Let $I$ be a positive integer-valued random variable, independent of $N$ and $\{X_i,X_i^*\}_{i\geq1}$, with
	\begin{equation}\label{Idef}
		P(I=i)=\frac{\sigma_i^2P(N'\geq i)}{\sigma^2},\qquad i\in\mathbb{N},
	\end{equation}
	where $N'$ is an independent copy of $N$. These probabilities sum to one by the expression for $\sigma^2$.
	
	\item Define
	\begin{equation}\label{Wsdef}
		W^*=
		\begin{cases}
			W-X_I+X_I^*, & \text{if }N\geq I,\\
			W, & \text{if }N<I.
		\end{cases}
	\end{equation}
\end{enumerate}

\begin{remark}
	When $N=n$ is fixed, $P(N\ge i)=1$ for $1\le i\le n$, and the index $I$ can be chosen with $P(I=i)=\sigma_i^2/\sigma^2$ on $[n]$. Consequently, only the first case in \eqref{Wsdef} applies, recovering the standard zero-bias construction (see, for instance, \citet{Ross11}). When $N$ is random, the same choice of probabilities need not define a probability distribution, since $\sigma^2=\sum_{i\ge1}\sigma_i^2P(N\ge i)$ rather than $\sum_{i\ge1}\sigma_i^2$. Our modified construction accommodates the random number of summands and leads to the remainder term shown below.
\end{remark}

We next show that $W^*$, defined in \eqref{Wsdef}, satisfies the zero-bias identity with an additional remainder term, as stated in \eqref{approxzero}. We therefore refer to $W^*$ as having an \emph{approximate zero-bias distribution} with respect to $W$. Such an identity does not uniquely determine the distribution of $W^*$ when the remainder is unrestricted. The objective is to construct a coupling for which the remainder can be controlled effectively in the normal approximation argument. Approximate zero-bias couplings were previously introduced in \citet{NW2017} through an approximate Stein coupling. 

\begin{proposition}\label{2.1}
	Let $X_1,X_2,\ldots$ be independent mean-zero random variables with finite variances $\Var(X_i)=\sigma_i^2$. Let $N$ be a nonnegative integer-valued random variable independent of $\{X_i\}_{i\geq1}$, and define $W=\sum_{i=1}^{N}X_i$, with the empty sum interpreted as zero. Assume that
	\begin{equation*}
		0<\sigma^2:=\Var(W)=\sum_{i=1}^{\infty}\sigma_i^2P(N\geq i)<\infty.
	\end{equation*}
	Let $X_i^z$, $I$, and $W^*$ be defined by the preceding construction, where $X_i^z$ has the zero-bias distribution of $X_i$ when $\sigma_i^2>0$, and $X_i^z=0$ otherwise. Then, for every absolutely continuous function $f$ with bounded derivative,
	\begin{equation}\label{approxzero}
		\E[Wf(W)]=\sigma^2\E[f'(W^*)]+R_f,
	\end{equation}
	where, if $P(N<I)>0$,
	\begin{align*}
		R_f=\sigma^2P(N<I)\Bigg\{
		&\sum_{i=1}^{\infty}P(I=i\mid N<I)
		\E\left[f'(W-X_i+X_i^*)\mid N\geq i\right] \nn \\
		&-\E\left[f'(W)\mid N<I\right]\Bigg\}.
	\end{align*}
	Terms corresponding to indices with $P(I=i)=0$ are interpreted as zero. If $P(N<I)=0$, set $R_f=0$.
\end{proposition}

\proof
Let $f$ be absolutely continuous with bounded derivative. Starting from the left-hand side, we obtain
\beas
\E Wf(W) &=& \E \left[\sum_{i=1}^N X_i f(W-X_i+X_i)\right] \\
&=& \sum_{k=0}^\infty \sum_{i=1}^k\E\left[ X_i f\left(\sum_{j=1}^kX_j-X_i+X_i\right)\right] P(N=k) \\
&=& \sum_{k=0}^\infty \sum_{i=1}^k \sigma_i^2 \E\left[ f'\left(\sum_{j=1}^kX_j-X_i+X_i^*\right)\right] P(N=k)  \\
&=& \sum_{i=1}^\infty \sigma_i^2 \sum_{k=i}^\infty \E\left[  f'\left(\sum_{j=1}^kX_j-X_i+X_i^*\right)\right] P(N=k) \\
&=& \sum_{i=1}^\infty \sigma_i^2 P(N\ge i) \sum_{k=i}^\infty \E\left[  f'\left(\sum_{j=1}^kX_j-X_i+X_i^*\right)\right] P(N=k|N\ge i) \\
&=& \sum_{i=1}^\infty \sigma^2  \E\left[  f'\left(\sum_{j=1}^NX_j-X_i+X_i^*\right)\Bigg|N\ge i\right]P(I=i)  \\
&=& \sum_{i=1}^\infty \sigma^2  \E\left[  f'\left(W-X_i+X_i^*\right)\Bigg|N\ge i\right]P(I=i|N\ge I)P(N\ge I)  \\
&&+\sum_{i=1}^\infty \sigma^2  \E\left[  f'\left(W-X_i+X_i^*\right)\Bigg|N\ge i\right]P(I=i|N< I)P(N< I) \\
&=& \sigma^2 \E\left[f'(W-X_I+X_I^*)|N\ge I\right]P(N\ge I) \\ 
&&+\sum_{i=1}^\infty \sigma^2  \E\left[  f'\left(W-X_i+X_i^*\right)\Bigg|N\ge i\right]P(I=i|N< I)P(N< I) \\
&=& \sigma^2 \E\left[f'(W^*)|N\ge I\right]P(N\ge I)\\
&&+\sum_{i=1}^\infty \sigma^2  \E\left[  f'\left(W-X_i+X_i^*\right)\Bigg|N\ge i\right]P(I=i|N< I)P(N< I).
\enas

Next, we handle the right hand side of \eqref{zerodef}. We have

\beas
\sigma^2\E[f'(W^*)] = \sigma^2\E\left[f'(W^*)|N\ge I\right]P(N\ge I)+\sigma^2\E\left[f'(W)|N< I\right]P(N< I).
\enas
Combining with the previous expression yields \eqref{approxzero}.

\bbox

\begin{corollary}\label{cor:stein}
	Under the assumptions of Proposition~\ref{2.1}, for every absolutely continuous function $f$ with bounded derivative,
	\begin{equation}\label{approxzero2}
		\E\left[\frac{W}{\sigma}f\left(\frac{W}{\sigma}\right)\right]
		=\E\left[f'\left(\frac{W^*}{\sigma}\right)\right]+\widetilde R_f,
	\end{equation}
	where, if $P(N<I)>0$,
	\begin{align}\label{rfs}
		\widetilde R_f=P(N<I)\Bigg\{
		&\sum_{i=1}^{\infty}P(I=i\mid N<I)
		\E\left[f'\left(\frac{W-X_i+X_i^*}{\sigma}\right)\Bigm|N\geq i\right]\nonumber\\
		&-\E\left[f'\left(\frac{W}{\sigma}\right)\Bigm|N<I\right]\Bigg\}.
	\end{align}
	Terms corresponding to indices with $P(I=i)=0$ are interpreted as zero. If $P(N<I)=0$, set $\widetilde R_f=0$.
\end{corollary}

\proof
Apply Proposition~\ref{2.1} to $g(x)=f(x/\sigma)$, noting that $g'(x)=f'(x/\sigma)/\sigma$ almost everywhere, and divide the resulting identity by $\sigma$.
\bbox

\bigskip

We now bound the probability 
$P(N<I)$ appearing in the remainder term. 

\begin{lemma}\label{prop:upperbound}
	Let \(N\) be a nonnegative integer-valued random variable with \(0<\E N<\infty\) and \(\Var(N)<\infty\), and let \(I\) be as in \eqref{Idef}. If \(0<\inf_i\sigma_i^2\le\sup_i\sigma_i^2<\infty\), then
	\bea\label{remainbound1}
	P(N<I)\le \frac{\sup_i\sigma_i^2\sqrt{\Var(N)}}{\sqrt{2}\inf_i\sigma_i^2\,\E N}.
	\ena
\end{lemma}

\proof
By \eqref{Idef}, we have
\beas
P(N\ge I) &=& \sum_{n=1}^\infty \sum_{i=1}^n P(N=n)P(I=i) = \sum_{n=1}^\infty \sum_{i=1}^n P(N=n)\frac{\sigma_i^2}{\sigma^2}P(N' \ge i) \\
&=& \sum_{i=1}^\infty \frac{\sigma_i^2}{\sigma^2}P(N' \ge i) \sum_{n=i}^\infty P(N=n) 
=\sum_{i=1}^\infty \frac{\sigma_i^2}{\sigma^2}P^2(N \ge i).
\enas
Therefore
\beas
P(N<I) &=& 1-\sum_{i=1}^\infty \frac{\sigma_i^2}{\sigma^2}P^2(N \ge i)  = \frac{1}{\sigma^2}\sum_{i=1}^\infty \sigma_i^2 (P(N\ge i) - P^2(N \ge i)) \\ 
&\le&  \frac{\sup_i\sigma_i^2}{\sum_{j=1}^\infty \sigma_j^2 P(N\ge j)}\sum_{i=1}^\infty P(N \ge i)P(N<i) \\
&\le& \frac{\sup_i \sigma_i^2}{\inf_i\sigma_i^2 \E N} \sum_{i=1}^\infty P(N \ge i)P(N<i).
\enas

Letting $N'$ be an independent copy of $N$ and applying Tonelli's theorem, we have 
\beas
\sum_{i=1}^{\infty}P(N\ge i)P(N<i)
&=& \E\left[\sum_{i=1}^{\infty}\mathbf{1}_{\{N'<i\le N\}}\right]\\
&=& \E[(N-N')_+]\\
&=& \frac12\E|N-N'|\\
&\le& \frac12\sqrt{\E[(N-N')^2]}
=\sqrt{\frac{\Var(N)}{2}},
\enas
where $x_+=\max\{x,0\}$, the third equality follows from the symmetry of $N-N'$, and the inequality follows from the Cauchy--Schwarz inequality. Substituting into the previous expression proves \eqref{remainbound1}.
\bbox

\begin{comment}
	Using the Chebyshev's inequality, for $\E N - (k+1)\sqrt{\Var(N)} < i \le \E N - k\sqrt{\Var(N)}$ with $k\ge 1$, we have
	\beas
	P(N < i ) \le \frac{1}{k^2} \text{ \ and thus \ } P(N \ge i)P(N<i) \le \frac{1}{k^2},
	\enas
	and for $\E N + k\sqrt{\Var(N)} \le i < \E N + (k+1)\sqrt{\Var(N)}$,
	\beas
	P(N \ge i ) \le \frac{1}{k^2} \text{ \ and thus \ } P(N \ge i)P(N<i) \le \frac{1}{k^2}.
	\enas
	\rcolor{We note that the number of $i$ in each interval is at most $\sqrt{\Var(N)}$ and the number of $i \in \left(\E N - \sqrt{\Var(N)},\E N +\sqrt{\Var(N)}\right)$ is at most $2\sqrt{\Var(N)}$ and that
		\beas
		P(N\ge i)P(N<i) \le \frac{1}{4} \text{ \ \ for all \ \ } i.
		\enas
	}
	\rcolor{
		Hence, using that $\sum_{k=1}^\infty\frac{1}{k^2} = \frac{\pi^2}{6}$,
		\beas
		P(N<I) \le \left(\frac{\pi^2}{3} + \frac{1}{2}\right)\frac{\max_i \sigma_i^2\sqrt{\Var(N)}}{\min_i\sigma_i^2 \E N} .
		\enas
	}
	\bbox    
\end{comment}

%%%%%%%%%%%%%%%%%%%%%%%%%%%%%%%%%%%%%%%%%%%%%%%%%%%%%%%%%%%%%%%%%%%%%%%%%%%%%%%%%

The following lemma specializes Lemma~\ref{prop:upperbound} to binomial, Poisson, and hypergeometric counts.

\begin{lemma}\label{cor:specialcase}
	Under the assumptions of Lemma~\ref{prop:upperbound}, the following bounds hold.
	\begin{enumerate}
		\item If $N\sim\mathrm{Bin}(n,p)$, where $n\in\mathbb{N}$ and $0<p<1$, then
		\beas
		P(N<I)\le \frac{\sup_i\sigma_i^2}{\inf_i\sigma_i^2}\sqrt{\frac{1-p}{2np}}.
		\enas
		
		\item If $N\sim\mathrm{Poisson}(\lambda)$, where $\lambda>0$, then
		\beas
		P(N<I)\le \frac{\sup_i\sigma_i^2}{\inf_i\sigma_i^2}\frac{1}{\sqrt{2\lambda}}.
		\enas
		
		\item If $N$ has a hypergeometric distribution with population size $M$, $K$ successes, and sample size $n$, where $M,K,n\in\mathbb{N}$, $M\ge2$, and $1\le K,n\le M$, then
		\beas
		P(N<I)\le \frac{\sup_i\sigma_i^2}{\inf_i\sigma_i^2}\sqrt{\frac{(M-K)(M-n)}{2nK(M-1)}}.
		\enas
	\end{enumerate}
\end{lemma}

The following theorem and its corollary provide bounds for a random sum and a standard normal. For a given  \(h \in \mathcal{H}_1\), let \(f_h \) denote the solution to the Stein's equation
\[
f_h'(x)-xf_h(x)=h(x)-\mathbb E[h(Z)],
\]
where \(Z\sim \mathcal{N}(0,1)\). For functions $f:\mathbb{R}\to\mathbb{R}$,
we let $\|f\|=\sup_{x\in\mathbb{R}}|f(x)|$
be the supremum norm.

\begin{theorem}\label{uniform}
	Let $X_1,X_2,\ldots$ be independent mean zero random variables such that $\Var(X_i) = \sigma_i^2$. Let $W=\sum_{i=1}^NX_i$, where $N$ is a discrete random variable with support being a subset of $\mathbb{N}_0$ and independent of all $X_i$.
	Define $\sigma^2 := \Var(W) = \sum_{i=1}^\infty\sigma_i^2 P(N\ge i)$. Let $I$ be as in \eqref{Idef}.  Then
	\begin{align*}
		\left|\E h\left(\frac{W}{\sigma}\right)-N(h)\right|\leq \frac{||f''_h||}{\sigma^3}\sum_{i=1}^\infty \left({1\over 2}\E|X_i|^3+ \sigma_i^2\E|X_i| \right)P^2(N\ge i)+2 ||f'_{h}|| \cdot {P(N < I)},
	\end{align*}
	where $N(h) = \E h(Z)$ with $Z \sim \mathcal{N}(0,1)$, $h \in \mathcal{H}_1$ and $f_h$ is the solution to the Stein's equation.
\end{theorem}

\proof 
By Stein's equation and Corollary \ref{cor:stein}, we have
\begin{align*}
	\left|\E h\left(\frac{W}{\sigma}\right)-N(h)\right|
	&=\left|\E\left[f_h'\left(\frac{W}{\sigma}\right)-\frac{W}{\sigma}f_h\left(\frac{W}{\sigma}\right)\right]\right|\\
	&\leq \left|\E\left[f_h'\left(\frac{W}{\sigma}\right)-f_h'\left(\frac{W^*}{\sigma}\right)\right]\right| + |\widetilde R_f|\\
	&\leq \frac{\|f_h''\|}{\sigma}\E|W^*-W| + |\widetilde R_f|
\end{align*}

Using the zero-bias identity with $f(x)=x^2$, which gives
$\sigma^2_i\E[X^z_i]=\E[X^3_i]/2$, the first term can be bounded as 
\beas
\E|W^*-W|&=& \E\left[|W^*-W| \mathbf{1}(N\ge I) \right]+\E\left[|W^*-W|\mathbf{1}(N< I)\right]\\
&=& \E\left[|X_I^z-X_I| \mathbf{1}(N\ge I) \right] = \sum_{i=1}^\infty\E\left[|X_i^z-X_i| \mathbf{1}(N\ge i) \right] \frac{\sigma_i^2P(N'\ge i)}{\sigma^2} \\
&=& \sum_{i=1}^\infty\E|X_i^z-X_i| \,P(N\ge i) \frac{\sigma_i^2P(N'\ge i)}{\sigma^2} 
= \frac{1}{\sigma^2}\sum_{i=1}^\infty \sigma_i^2\E|X_i^z-X_i| P^2(N\ge i)\\
&\le & \frac{1}{\sigma^2}\sum_{i=1}^\infty \sigma_i^2\left(\E|X_i^z|+\E|X_i| \right)P^2(N\ge i)\\
&\le & \frac{1}{\sigma^2}\sum_{i=1}^\infty \left({1\over 2 }\E|X_i|^3+ \sigma_i^2\E|X_i| \right)P^2(N\ge i) .
\enas 

The second term can be bounded as 
\beas
|\widetilde R_f| \le \| f_h'\| P(N<I) \left(\sum_{i=1}^\infty   P(I=i|N< I) + 1 \right) \le 2\| f_h'\| P(N<I)
\enas
Combining both terms completes the proof.
\bbox

\begin{corollary} \label{coro2.5} Let $X_1,X_2,\ldots$ be independent mean zero random variables such that $\Var(X_i) = \sigma_i^2$. Let $W=\sum_{i=1}^NX_i$, where $N$ is  a discrete random variable with support being subset of $\mathbb{N}_0$, independent of $\{X_i\}_{i=1}^\infty$.
	Define $\sigma^2 := \Var(W) = \sum_{i=1}^\infty\sigma_i^2 P(N\ge i)$. Let $I$ be as in \eqref{Idef}. Then 
	$$d_1\left(\mathcal{L}(W/\sigma), \mathcal{L}(Z) \right)\le \frac{1}{\sigma^3}\sum_{i=1}^\infty \left(\E|X_i|^3+ 2 \sigma_i^2\E|X_i| \right)P^2(N\ge i)+2 \sqrt{\frac{2}{\pi}}\cdot\displaystyle P(N<I),$$  where $Z$ is a standard normal random variable.
\end{corollary}
\proof 
For a given  \(h \in \mathcal{H}_1\), let \(f_h \) denote the solution to the  Stein's equation
\[
f_h'(x)-xf_h(x)=h(x)-\mathbb E[h(Z)],
\]
where \(Z\sim N(0,1)\). By Lemma 2.4 of \cite{GR05}, we have
\[
\|f_h\|\le 2,\qquad
\|f_h'\|\le\sqrt{\frac{2}{\pi}},\qquad
\|f_h''\|\le 2.
\] 
Applying Theorem Theorem~\ref{uniform} and using
the above Stein factor bounds, we obtain 

\begin{align*}
	d_1\left(\mathcal{L}(W/\sigma), \mathcal{L}(Z) \right)&= \sup_{h \in \mathcal{H}_1} |\E h(W/\sigma)-N(h)| \\
	&\le \sup_{f_h} ||f''_h||{\sigma^3}\sum_{i=1}^\infty \left({1\over 2}\E|X_i|^3+ \sigma_i^2\E|X_i| \right)P^2(N\ge i)+2 ||f_{h}'|| \cdot{P(N < I)}\\
	&\le \frac{1}{\sigma^3}\sum_{i=1}^\infty \left(\E|X_i|^3+ 2 \sigma_i^2\E|X_i| \right)P^2(N\ge i)+2 \sqrt{\frac{2}{\pi}}\cdot\displaystyle P(N<I).
\end{align*}

\bbox

%%%%%%%%%

\section{Applications} \label{sec:app}

In this section, we illustrate our theoretical results through models motivated by sampling with outliers, insurance claims, and generative AI response times. The final subsection compares our bounds with existing bounds in the i.i.d. case.

\subsection{Simple random sampling with outliers} \label{1stapp}

In this subsection, we study normal approximation for the sum of non-outlier values in a simple random sample from a population containing a small proportion of outliers. Although the sample size is fixed, the number of sampled non-outliers is random and follows a hypergeometric distribution. We adopt a random-population model in which the non-outlier values are generated independently from a common distribution, and sampling is independent of these values. This allows the sum of sampled non-outlier values to be represented as a random sum with an independent number of summands, to which our approximation results apply.

\subsubsection{Theoretical results}

Let $m,N_0\in\mathbb{N}$ satisfy $m\ll N_0$. Consider a population of $N_0+m$ labeled units, consisting of $N_0$ non-outliers and $m$ outliers. We assume that the non-outlier values are generated independently from a common distribution with mean zero and variance $\sigma_1^2>0$. Draw a simple random sample of size $n$, where $1\le n\le N_0+m$, without replacement, independently of all population values. Our analysis concerns the unconditional distribution, averaging over both population generation and sampling.

Let $N$ denote the number of non-outliers in the sample. Then $N$ has a hypergeometric distribution with population size $N_0+m$, $N_0$ successes, and sample size $n$. Thus,
\begin{equation*}
	N\sim H(N_0+m,n,N_0),
	\qquad
	\E[N]=\frac{nN_0}{N_0+m},
\end{equation*}
with support $\max \{0,n-m\}\le N\le\min \{n,N_0\}$.

Conditional on $N=k$, the sampled non-outlier values, listed in an order determined independently of their values, are $k$ independent observations from the common non-outlier distribution. Consequently, their sum can be represented in distribution as
\begin{equation*}
	W=\sum_{i=1}^{N}X_i,
\end{equation*}
where $X_1,X_2,\ldots$ are i.i.d. with the common non-outlier distribution and are independent of $N$. The sum is interpreted as zero when $N=0$. Since the $X_i$ are independent, have mean zero and variance $\sigma_1^2$, and are independent of $N$, the variance of $W$ is
\begin{align}\label{SRS:sigma}
	\sigma^2=\Var(W)
	=\sigma_1^2\E[ N]
	=\sigma_1^2\sum_{i=1}^{\min\{n,N_0\}}P(N\ge i)
	=\frac{nN_0\sigma_1^2}{N_0+m}.
\end{align}

The following theorem provides an upper bound for the Wasserstein distance between $W/\sigma$ and a standard normal random variable.

\begin{theorem} \label{SRM:Theoretical result} Let $N_0,m,n\in\mathbb{N}$ satisfy $1\le n\le N_0+m$. Let $X_1,X_2,\ldots$ be i.i.d.\ mean-zero random variables with $\Var(X_1)=\sigma_1^2>0$ and $\E|X_1|^3<\infty$. Let $N$, independent of $\{X_i\}_{i\ge1}$, follow a hypergeometric distribution with population size $N_0+m$, $N_0$ successes, and sample size $n$. Define $W=\sum_{i=1}^{N}X_i$, with the empty sum interpreted as zero, and let $\sigma^2=\Var(W)$ be given by \eqref{SRS:sigma}. Then 
	\bea \label{Eqn: Hypergeometric UB}
	d_1\left(\mathcal{L}(W/\sigma), \mathcal{L}(Z) \right)  \le &\left({\E|X_1|^3 \over \left({nN_0 \sigma_1^2\over N_0+m}\right)^{3/2}} +{2\sigma_1^2\E|X_1|\over \left({nN_0\sigma_1^2\over N_0+m}\right)^{3/2}}\right)\sum_{i=1}^{\min\{n,N_0\}} P^2(N\ge i)  \nn \\
	&\quad \quad+{2\sqrt{\frac{2}{\pi}}}\cdot\displaystyle{\left(1- \left( \frac{N_0+m}{nN_0}\right)\sum_{i=1}^{\min\{n,N_0\}} P^2(N\ge i)\right)}, \ena  where $Z$ is a standard normal random variable. In particular,
	\bea  \label{Eqn: Hypergeometric explicit UB}
	d_1\left(\mathcal{L}(W/{\sigma}), \mathcal{L}(Z) \right)&\le \frac{\sqrt{N_0+m} }{ \sigma_1^3\sqrt{nN_0} } \left(\E|X_1|^3+ 2 \sigma_1^2\E|X_1| \right)+\frac{2\sqrt{m(N_0+m-n)}}{ \sqrt{\pi nN_0(N_0+m-1)}}.
	\ena
\end{theorem}
\proof Let $I$ be defined as in \eqref{Idef}. By the independence of $I$ and $N$ and \eqref{SRS:sigma},
\begin{align*}
	P(N \ge I) &=  \sum_{i=1}^{\min\{n,N_0\}} \frac{\Var(X_i)}{\sigma^2} \, P^2(N \ge i)\\
	&=  \sum_{i=1}^{\min\{n,N_0\}} \frac{\sigma_1^2}{\sigma^2} \, P^2(N \ge i)\\& =  \left( \frac{N_0+m}{nN_0}\right)\sum_{i=1}^{\min\{n,N_0\}} \, P^2(N \ge i).
\end{align*} 
Applying Corollary~\ref{coro2.5} yields \eqref{Eqn: Hypergeometric UB}. Furthermore,
\begin{equation*}
	\sum_{i=1}^{\min\{n,N_0\}}P^2(N\ge i)
	\le\sum_{i=1}^{\min\{n,N_0\}}P(N\ge i)
	=\E[N]=\frac{nN_0}{N_0+m},
\end{equation*}
and
\begin{equation*}
	\Var(N)=\frac{nN_0m(N_0+m-n)}{(N_0+m)^2(N_0+m-1)}.
\end{equation*}
Using these identities and applying Lemma~\ref{cor:specialcase} to bound $P(N<I)$ yields \eqref{Eqn: Hypergeometric explicit UB}.
\bbox

\subsubsection{Simulation}

We apply Theorem~\ref{SRM:Theoretical result} when the non-outlier values follow an exponential distribution with mean $\lambda>0$. Let $Y_1,Y_2,\ldots$ be i.i.d.\ with this distribution and set $X_i=Y_i-\lambda$. The corresponding absolute moments are
\begin{align}\label{absolute central moments}
	\E|X_i|=\frac{2\lambda}{e},\qquad
	\E|X_i|^2=\lambda^2,\qquad
	\E|X_i|^3=\lambda^3\left(\frac{12}{e}-2\right).
\end{align}
Substituting these moments into the first bound in Theorem~\ref{SRM:Theoretical result} yields the upper bounds reported below, which are independent of $\lambda$.

We compare these upper bounds with simulation-based estimates of $d_1(\mathcal{L}(W/\sigma),\mathcal{L}(Z))$. We set $N_0=100{,}000{,}000$, $m=1{,}200{,}000$, and $\lambda=5$. The non-outlier values are generated independently from the exponential distribution and centered by subtracting $\lambda$. The outlier values are generated from the mixture
\begin{equation*}
	O\sim
	\begin{cases}
		\mathrm{Unif}(-20,-10), & \text{with probability } 55/70, \\
		\mathrm{Unif}(50,60), & \text{with probability } 15/70,
	\end{cases}
\end{equation*}
which satisfies $\E[O]=0$. Since $W$ includes only the centered non-outlier values, the outlier distribution does not affect its distribution or the approximation bound; the outlier count $m$ enters through the distribution of $N$.

Table~\ref{tab:samp} reports the theoretical upper bounds and the estimated distances based on $100{,}000$ independent simulation replications, with randomness taken over both population generation and sampling.

\begin{table}[!htbp]
	\centering
	\begin{tabular}{|c|c|c|c|c|c|}
		\hline
		$n$ & Upper bound  & Simulated distance
		& $n$ & Upper bound & Simulated distance   \\
		\hline
		$2^1$   & 2.75061143 & 0.19881889 & $2^{9}$  & 0.17661563 & 0.01431534 \\
		$2^2$   & 1.95054948 & 0.15390773 & $2^{10}$ & 0.12499765 & 0.01214080 \\
		$2^3$   & 1.38449178 & 0.11342520 & $2^{11}$ & 0.08844060 & 0.00853899 \\
		$2^4$   & 0.98356093 & 0.08105740 & $2^{12}$ & 0.06256320 & 0.00800496 \\
		$2^5$   & 0.69896369 & 0.05700365 & $2^{13}$ & 0.04425172 & 0.00759138 \\
		$2^6$   & 0.49638711 & 0.04091913 & $2^{14}$ & 0.03129700 & 0.00484885 \\
		$2^7$   & 0.35206696 & 0.02958365 & $2^{15}$ & 0.02213342 & 0.00400173 \\
		$2^8$   & 0.24944274 & 0.01988771 & $2^{16}$ & 0.01565219 & 0.00877013 \\
		\hline
	\end{tabular}
	\caption{The comparison of theoretical upper bounds and simulated distances for different values of $n$. The simulated distance is computed from 100,000 repeated simulations.}
	\label{tab:samp}
\end{table}

\begin{figure}[!htbp]
	\centering
	\includegraphics[width=0.8\textwidth]{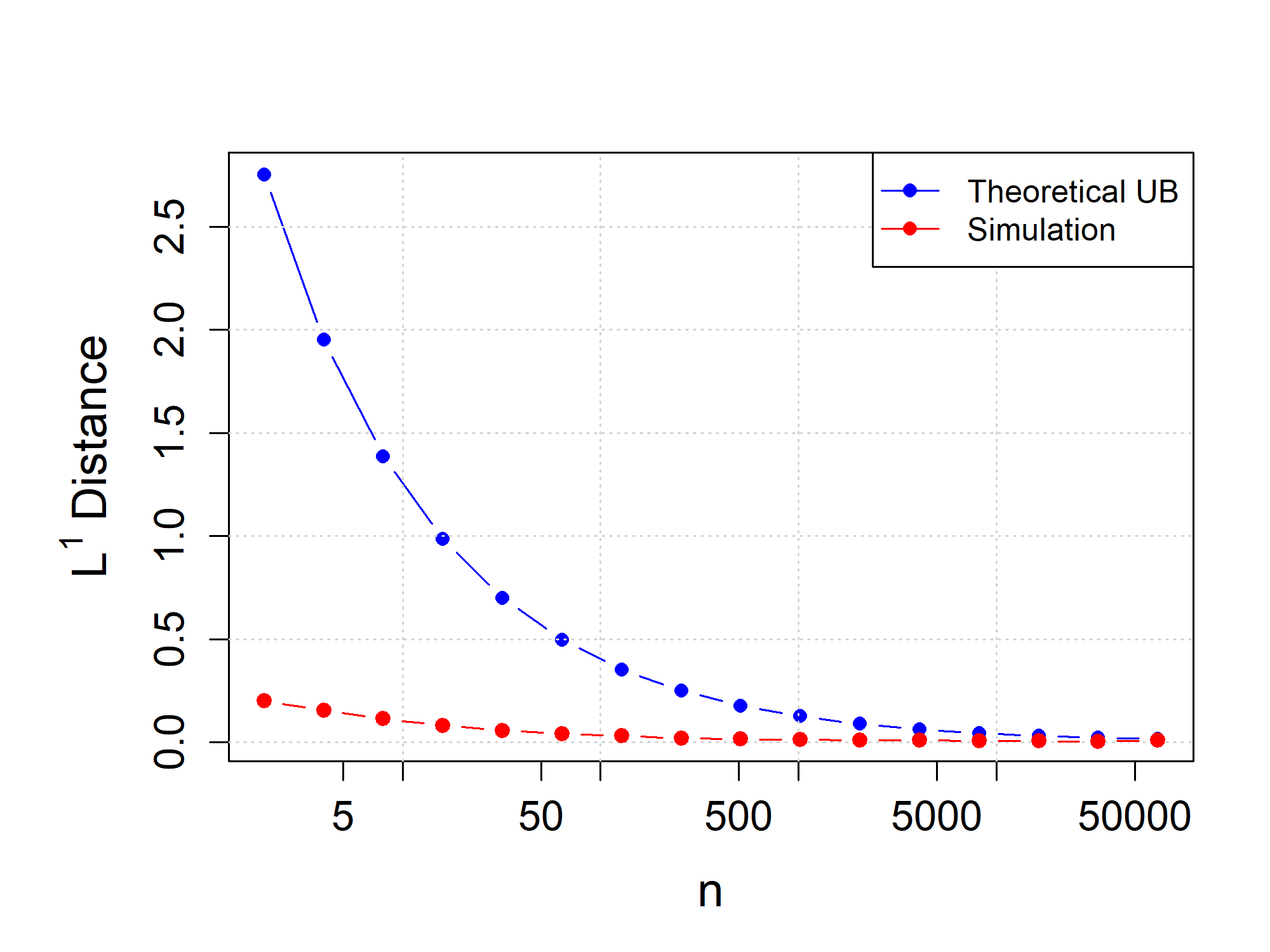}
	\caption{Simulated and theoretical upper bounds as functions of $n$, plotted on a logarithmic scale with a consistent $x$-axis spacing.}
	\label{fig:placeholder}
\end{figure}

Figure~\ref{fig:placeholder} compare the theoretical upper bounds (blue) with the simulation-based estimates of $d_1(\mathcal{L}(W/\sigma),\mathcal{L}(Z))$ for sample sizes $n$ ranging from $2$ to $2^{16}$. The estimated distances remain below the theoretical bounds for all sample sizes considered. Both decrease at comparable rates, although the bounds are consistently larger, reflecting the conservative constants in the theoretical estimates.

The simulated distances decrease steadily until $n$ exceeds approximately $1,000$, at which point they reach the order of $10^{-2}$ and exhibit fluctuations, including occasional increases. This behavior is consistent with the numerical limitations of the estimation procedure: the distances are estimated using a finite number of Monte Carlo replications and a grid of 5,000 points. When the underlying distance becomes small, Monte Carlo and grid approximation errors can obscure its further decrease. The theoretical upper bounds remain larger because of their conservative constants. Increasing the number of replications and refining the grid could improve numerical accuracy, but at a substantial computational cost. These limitations also apply to the simulations in the remaining sections.

\subsection{Auto Insurance claims}

In this subsection, we study normal approximation for the aggregate amount of auto insurance claims using synthetic data. We model individual claim amounts by a Gamma distribution and the number of claims by a Binomial distribution, with parameters estimated by maximum likelihood from publicly available data on Kaggle \citep{Xiaomeng}.

\subsubsection{Theoretical results}

Consider $n$ policyholders, each independently making at least one claim during a fixed observation period with probability $p\in(0,1]$. The number of policyholders making claims is therefore $N\sim\mathrm{Binomial}(n,p)$. Let $X_1,X_2,\ldots$ be i.i.d.\ centered total claim amounts for these policyholders, with variance $0<\sigma_1^2<\infty$, and assume that this sequence is independent of $N$. Define
\beas
W=\sum_{i=1}^{N}X_i,
\enas
with the empty sum interpreted as zero. Then
\begin{align}\label{Auto Insurance:sigma}
	\sigma^2=\Var(W)=\sigma_1^2\E[N]=\sigma_1^2np.
\end{align}
Here, $W$ represents the deviation of the aggregate claim amount from its conditional expectation given $N$.

The following theorem specializes our normal approximation bound to i.i.d.  summands and a binomial number of terms.

\begin{theorem}\label{thm:BN}
	Let $X_1,X_2,\ldots$ be i.i.d.\ mean-zero random variables with $\Var(X_1)=\sigma_1^2>0$ and $\E|X_1|^3<\infty$. Let $N\sim\mathrm{Binomial}(n,p)$, where $n\in\mathbb{N}$ and $p\in(0,1]$, be independent of $\{X_i\}_{i\ge1}$. Define $W=\sum_{i=1}^N X_i$, with the empty sum interpreted as zero, and let $\sigma^2=\Var(W)=\sigma_1^2np$. Let $I$ be defined as in \eqref{Idef}. Then
	\begin{equation}\label{Eqn: Binomial UB}
		d_1\left(\mathcal{L}(W/\sigma),\mathcal{L}(Z)\right)
		\le
		\frac{\E|X_1|^3+2\sigma_1^2\E|X_1|}{\sigma_1^3(np)^{3/2}}
		\sum_{i=1}^{n}P^2(N\ge i)
		+2\sqrt{\frac{2}{\pi}}P(N<I),
	\end{equation}
	where $Z$ is a standard normal random variable. In particular,
	\begin{equation}\label{Eqn: Binomial explicit UB}
		d_1\left(\mathcal{L}(W/\sigma),\mathcal{L}(Z)\right)
		\le
		\frac{\E|X_1|^3+2\sigma_1^2\E|X_1|}{\sigma_1^3\sqrt{np}}
		+2\sqrt{\frac{1-p}{\pi np}}.
	\end{equation}
\end{theorem}

\proof 
By \eqref{Idef} and the independence of $I$ and $N$,
\begin{equation*}
	P(N\ge I)=\frac{1}{np}\sum_{i=1}^{n}P^2(N\ge i).
\end{equation*}
Applying Corollary~\ref{coro2.5} yields \eqref{Eqn: Binomial UB}. Since
\begin{equation*}
	\sum_{i=1}^{n}P^2(N\ge i)\le\sum_{i=1}^{n}P(N\ge i)=np,
\end{equation*}
and $\Var(N)=np(1-p)$, applying Lemma~\ref{cor:specialcase} to bound $P(N<I)$ yields \eqref{Eqn: Binomial explicit UB}.
\bbox

\subsubsection{Simulation}

The original dataset \citep{Xiaomeng} contains 10,302 rows, of which 2,746 record positive claim amounts. These rows correspond to 8,753 unique policyholders, some of whom appear in multiple rows. After summing claim amounts across rows belonging to the same policyholder, we obtain 2,583 policyholders with positive total claim amounts and 6,170 with zero total claim amounts. Fitting a Gamma distribution to the positive policyholder totals by maximum likelihood yields shape parameter $k=0.8203$ and scale parameter $\theta=10962.0161$. The estimated probability of a policyholder making at least one claim is $p=2583/8753$.

For each portfolio of size $n$, we generate $N\sim\mathrm{Binomial}(n,p)$ and an independent sequence of i.i.d.\ positive policyholder claim totals $X'_1,X'_2,\ldots\sim\mathrm{Gamma}(k,\theta)$, using the fitted parameters above. We set $X_i=X'_i-k\theta$ and construct $W=\sum_{i=1}^{N}X_i$ as in the theoretical setup. Since centering does not change the variance, $\Var(X_i)=\Var(X'_i)=k\theta^2$. Hence, by \eqref{Auto Insurance:sigma}, $\sigma^2=k\theta^2np$.

We evaluate the first and third absolute central moments of the fitted Gamma distribution numerically and use its variance $k\theta^2$. Substituting these quantities into the first bound in Theorem~\ref{thm:BN} yields the theoretical upper bounds reported in Table~\ref{tab:insure}. For each portfolio size $n$, we generate $1{,}000{,}000$ independent realizations of $W/\sigma$ and use them to estimate $d_1(\mathcal{L}(W/\sigma),\mathcal{L}(Z))$.

\begin{table}[htbp]
	\centering
	\begin{tabular}{|c|c|c|c|c|c|}
		\hline
		$n$ &  Upper bounds & Simulated distance & $n$ &  Upper bounds & Simulated distance\\
		\hline
		$2^{1}$  & 3.182252156 & 0.397469296  & $2^{11}$ & 0.189724807 &  0.014580690\\
		$2^{2}$  & 2.805325953 & 0.270339191  & $2^{12}$ & 0.134803390 & 0.010225892\\
		$2^{3}$  & 2.292840787 & 0.216659633 & $2^{13}$ & 0.095644132 & 0.007348849 \\
		$2^{4}$  & 1.782629262 & 0.160307504 & $2^{14}$ & 0.067792444 & 0.005204517 \\
		$2^{5}$  & 1.342652002 & 0.115341608 & $2^{15}$ & 0.048017396 & 0.003791629 \\
		$2^{6}$  & 0.990774915 & 0.081930908 & $2^{16}$ & 0.033993871 & 0.002766599 \\
		$2^{7}$  & 0.721327391 & 0.058025642 & $2^{17}$ & 0.024057517 & 0.002042665  \\
		$2^{8}$  & 0.520433949 & 0.040943521 & $2^{18}$ & 0.017021343 & 0.001708912 \\
		$2^{9}$  & 0.373190694 & 0.029063853 & $2^{19}$ & 0.012040961 & 0.001505429 \\
		$2^{10}$ & 0.266478699 & 0.020438395 & $2^{20}$ & 0.008516772 & 0.001407022 \\
		\hline
	\end{tabular}
	
	\caption{The comparison of theoretical upper bounds and simulated distances for different values of $n$. The simulated distance is computed from 1,000,000 repeated simulations.}
	\label{tab:insure}
\end{table}

\begin{figure}[!htbp]  % 'h' means here, but you can use 't' for top, 'b' for bottom, etc.
	\centering
	\includegraphics[width=0.8\textwidth]{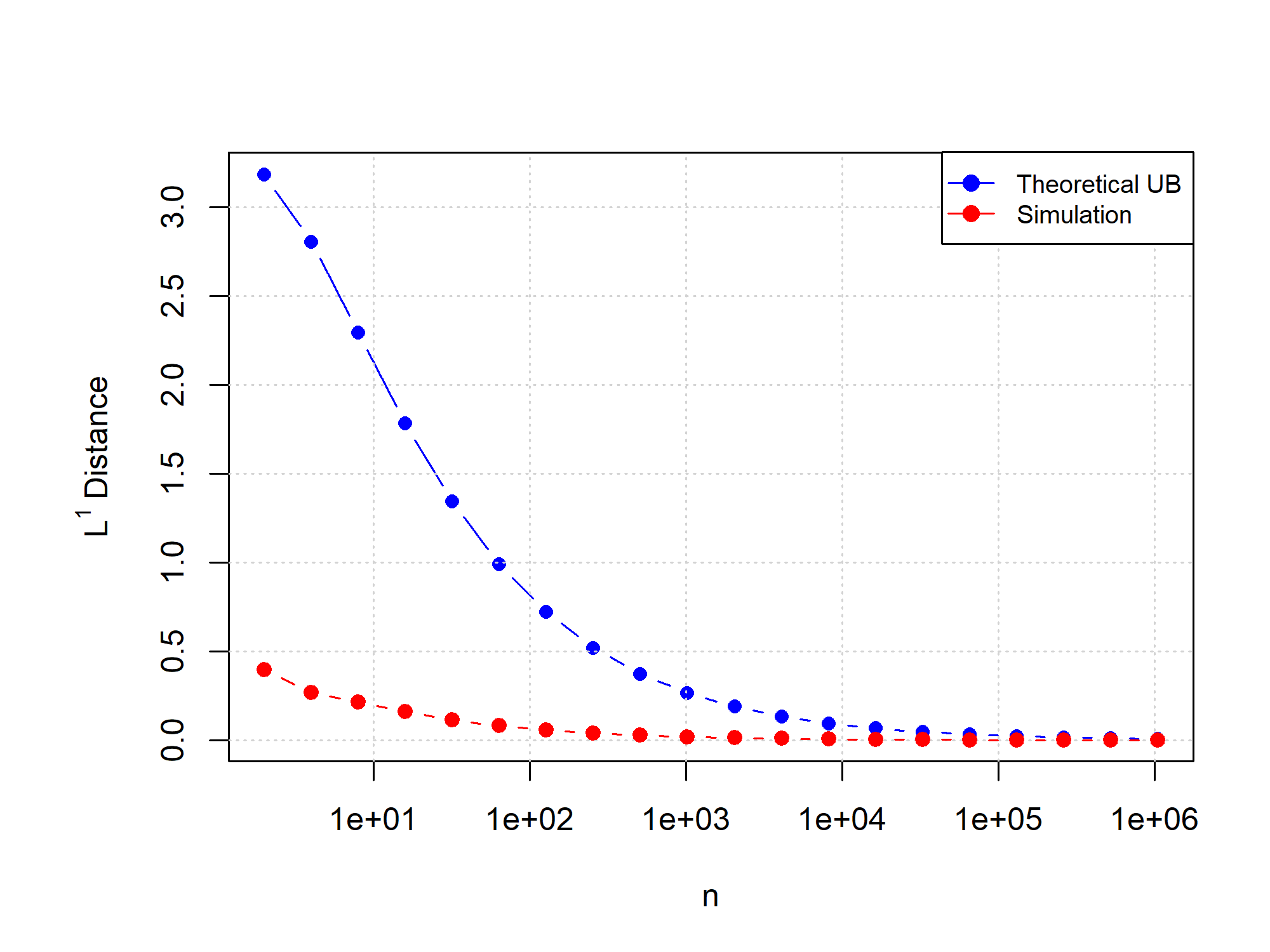}
	\caption{Simulated and theoretical upper bounds as functions of $n$, plotted on a logarithmic scale with a consistent $x$-axis spacing.}
	\label{fig:insure}
\end{figure}

Figure~\ref{fig:insure} compares the theoretical upper bounds (blue) with the simulation-based estimates of $d_1(\mathcal{L}(W/\sigma),\mathcal{L}(Z))$ (red). Both decrease as $n$ increases, with the estimated distances remaining below the theoretical bounds throughout. The decreasing distances indicate that the normal approximation improves as the portfolio size grows, while the absolute gap between the bounds and the estimates also narrows.

\subsection{Extension of auto insurance claims}\label{app}

In this subsection, we extend the preceding analysis to an illustrative model in which an insurer applies stricter review procedures after the number of claims reviewed reaches a prescribed threshold. We assume that this stricter review reduces the mean approved payment for subsequent claims. This motivates a two-group model with different claim-amount distributions before and after the threshold, allowing us to apply our normal approximation results to non-identically distributed summands. The change in review procedures is a modeling assumption rather than an empirically established feature of the dataset.

\subsubsection{Theoretical results}

Let $n_1\in\mathbb{N}$ denote the threshold at which stricter review begins. Let $X_1,X_2,\ldots$ be independent mean-zero random variables with $0<\Var(X_i)=\sigma_i^2<\infty$
for all $i\ge 1$, identically distributed within each of the two groups $X_1,\ldots,X_{n_1}$ and $X_{n_1+1},X_{n_1+2},\ldots$, with distributions that may differ between groups. Here, $X_i$ represents the total approved payment for the $i$th claiming policyholder in review order, centered by its group-specific mean.

Let $N\sim\mathrm{Binomial}(n,p)$ be independent of the sequence, where $n$ is the number of policyholders and $p$ is the probability that a policyholder makes at least one claim during the observation period. Define
\begin{equation*}
	W=\sum_{i=1}^{N}X_i,
\end{equation*}
with the empty sum interpreted as zero. Thus, $W$ represents the deviation of the aggregate approved payment from its conditional expectation given $N$.

Assume $1\le n_1<n$, and set $n_0=0$ and $n_2=n$. Then
\begin{equation}\label{TH:IC1}
	\sigma^2=\Var(W)
	=\sum_{k=1}^{2}\sigma_{n_{k-1}+1}^2
	\sum_{i=n_{k-1}+1}^{n_k}P(N\ge i).
\end{equation}

\begin{remark}
	If $n_1\ge n$, only the first group contributes to $W$, and the model reduces to the i.i.d.\ setting of the preceding subsection.
\end{remark}

The following theorem gives a Wasserstein distance bound for the normal approximation of $W/\sigma$ in this two-group setting.

\begin{theorem}
	\label{Thm:insurance} Let $n_0=0$, $n_2=n$, and $n_1\in\mathbb{N}$
	satisfy $1\le n_1<n$. Let $X_1,X_2,\ldots$ be mean zero random variables
	with $\Var(X_i)=\sigma_i^2>0$ and $\E|X_i|^3<\infty$
	for all $i\ge 1$.
	Suppose that
	$X_1,\ldots,X_{n_1} \text{ and }
	X_{n_1+1}, X_{n_1+2},\ldots$
	form two independent groups of i.i.d.\ random variables, where the distributions may differ
	between the two groups. Let $W=\sum_{i=1}^N X_i$, where $N$ is  a binomial random variable  with parameters $n$ and $p$, independent of $\{X_i\}_{i=1}^\infty$, and let $\sigma^2$ be defined as in \eqref{TH:IC1}.  %Let $I$ be as in \eqref{Idef}.
	Then
	\begin{align} \label{Eqn:Extension of insurance}
		d_1\left(\mathcal{L}(W/\sigma), \mathcal{L}(Z) \right)
		&\le \frac{1}{\sigma^3}
		\sum_{k=1}^{2}
		\left(
		\mathbb{E}|X_{n_{k-1}+1}|^3
		+2\sigma^2_{n_{k-1}+1}
		\mathbb{E}|X_{n_{k-1}+1}|
		\right)
		\sum_{i=n_{k-1}+1}^{n_k}
		P^2(N\ge i)\nn\\
		&\quad + 2\sqrt{\frac{2}{\pi}} \left(1 - \sum_{k=1}^{2}\frac{\sigma^2_{n_{k-1}+1}}{\sigma^2}\sum_{i=n_{k-1}+1}^{n_k}P^2(N\ge i)\right),
	\end{align}
	where $Z$ is a standard normal random variable. In particular,
	\begin{align}\label{Eqn: explicit UB}
		d_1\left(\mathcal{L}(W/\sigma),\mathcal{L}(Z)\right)
		&\le
		\frac{\max\limits_{1\le k\le 2}\{\E|X_{n_{k-1}+1}|^3\}+2\max\limits_{1\le k\le 2}\{\sigma^2_{n_{k-1}+1}\E|X_{n_{k-1}+1}| \}}{\min\limits_{1\le k\le 2}\{\sigma^3_{n_{k-1}+1}\}\sqrt{np}}\nn\\
		& \hspace{10pt}+ \frac{2\max\limits_{1\le k\le 2}\{\sigma^2_{n_{k-1}+1}\}}{\min\limits_{1\le k\le 2}\{\sigma^2_{n_{k-1}+1}\}}
		\sqrt{\frac{1-p}{\pi np}}.
	\end{align}    
\end{theorem}

\proof 
By \eqref{Idef} and the independence of $I$ and $N$,
\begin{equation}\label{TH:IC2}
	P(N\ge I)
	=\frac{1}{\sigma^2}
	\sum_{k=1}^{2}\sigma^2_{n_{k-1}+1}
	\sum_{i=n_{k-1}+1}^{n_k}P^2(N\ge i).
\end{equation}
Applying Corollary \ref{coro2.5} and using \eqref{TH:IC1} and \eqref{TH:IC2} yields \eqref{Eqn:Extension of insurance}.

By \eqref{TH:IC1}, 
\[
\sigma^2  \ge \min\limits_{1\le k\le 2} \{\sigma^2_{n_{k-1}+1}\} \sum_{i=1}^{n}P(N\ge i)= np \min\limits_{1\le k\le 2} \{\sigma^2_{n_{k-1}+1}\}.
\]
and
\begin{equation*}
	\sum_{i=1}^{n}P^2(N\ge i)\le\sum_{i=1}^{n}P(N\ge i)=np,
\end{equation*}
this gives the first term in \eqref{Eqn: explicit UB} by bounding each group-specific moment contribution by its
maximum over the two groups.

Since $\Var(N)=np(1-p)$, applying the Binomial case in Lemma~\ref{cor:specialcase} to bound $P(N<I)$ yields the second term in \eqref{Eqn: explicit UB}.
\bbox

\subsubsection{Simulation}

We retain \(N\sim\mathrm{Binomial}(n,p)\) with \(p=2583/8753\) from the preceding subsection. Independently of \(N\), we generate mutually independent approved payment amounts \(X'_1,X'_2,\ldots\), each representing the total payment for a claiming policyholder. The first \(n_1\) amounts follow a Gamma distribution with shape \(k_1=4.04\) and scale \(\theta_1=1025.78\), while subsequent amounts follow a Gamma distribution with shape \(k_2=2.5\) and scale \(\theta_2=800\). These parameters are chosen for illustration, giving mean payments of \(k_1\theta_1\approx4144.15\) and \(k_2\theta_2=2000\), respectively, consistent with the assumed reduction in mean approved payments after the review threshold. We center each amount by its group-specific mean and construct \(W\) as in the theoretical setup.
Since the variance of a Gamma random variable with shape \(k_j\) and scale \(\theta_j\) is \(k_j\theta_j^2\), equation~\eqref{TH:IC1} gives
\begin{equation*}
	\sigma^2=\sum_{j=1}^{2}k_j\theta_j^2
	\sum_{i=n_{j-1}+1}^{n_j}P(N\ge i).
\end{equation*}
We evaluate the first and third absolute central moments numerically for each group and use the variances \(k_j\theta_j^2\). Substituting these quantities into the first bound in Theorem~\ref{Thm:insurance} yields the theoretical upper bounds reported in Table~\ref{tab:insure2}, which we compare with simulation-based estimates of \(d_1(\mathcal{L}(W/\sigma),\mathcal{L}(Z))\).

\begin{table}[!htbp]
	\centering
	\begin{tabular}{|c|c|c|c|c|c|}
		\hline
		$n$ & Upper bounds & Simulated distance & $n$ & Upper bounds & Simulated distance \\
		\hline
		$2^1$  & 2.837730952 & 0.336570427 & $2^{11}$ & 0.173259958& 0.008203795 \\
		$2^2$  & 2.483412912 & 0.199716890 & $2^{12}$ & 0.122745093 & 0.005945545 \\
		$2^3$  & 2.020426998 & 0.110711251 & $2^{13}$ & 0.086570289 & 0.004286396  \\
		$2^4$  & 1.566328648 & 0.077239348 & $2^{14}$ & 0.061084707 & 0.003213821 \\
		$2^5$  & 1.177558296 & 0.053436330 & $2^{15}$ & 0.043149835 & 0.002353975 \\
		$2^6$  & 0.867886841 & 0.037369233 & $2^{16}$ & 0.030503179 & 0.001809013 \\
		$2^7$  & 0.631338871 & 0.026270865 & $2^{17}$ & 0.021570680 & 0.001647794 \\
		$2^8$  & 0.455242470 & 0.018602045 & $2^{18}$ & 0.015255875 & 0.001461229 \\
		$2^9$  & 0.317432568 & 0.014071492 & $2^{19}$ & 0.010789914 & 0.001391717 \\
		$2^{10}$ & 0.239888497 & 0.011096590 & $2^{20}$ & 0.007631111 & 0.001401313 \\
		\hline
	\end{tabular}
	\caption{The comparison of theoretical upper bounds and simulated distances for different values of $n$. The simulated distance is computed from 1,000,000 repeated simulations.}
	\label{tab:insure2}
\end{table}

Table~\ref{tab:insure2} and Figure~\ref{fig:insure2} compare the theoretical upper bounds with the simulation-based estimates of \(d_1(\mathcal{L}(W/\sigma),\mathcal{L}(Z))\). The estimated distances remain below the theoretical bounds for all values of \(n\) considered. Both decrease as \(n\) increases over most of the range, indicating an improvement in the normal approximation despite the change in the claim-amount distribution after the review threshold. For large \(n\), the theoretical bounds continue to decrease, whereas the estimated distances level off near \(1.4\times10^{-3}\) and increase slightly at the largest sample size. This behavior is consistent with the numerical limitations discussed in Section \ref{1stapp}. The absolute gap between the bounds and the estimated distances narrows substantially, although the bounds remain conservative.

\begin{figure}[H]  
	\centering
	\includegraphics[width=0.8\textwidth]{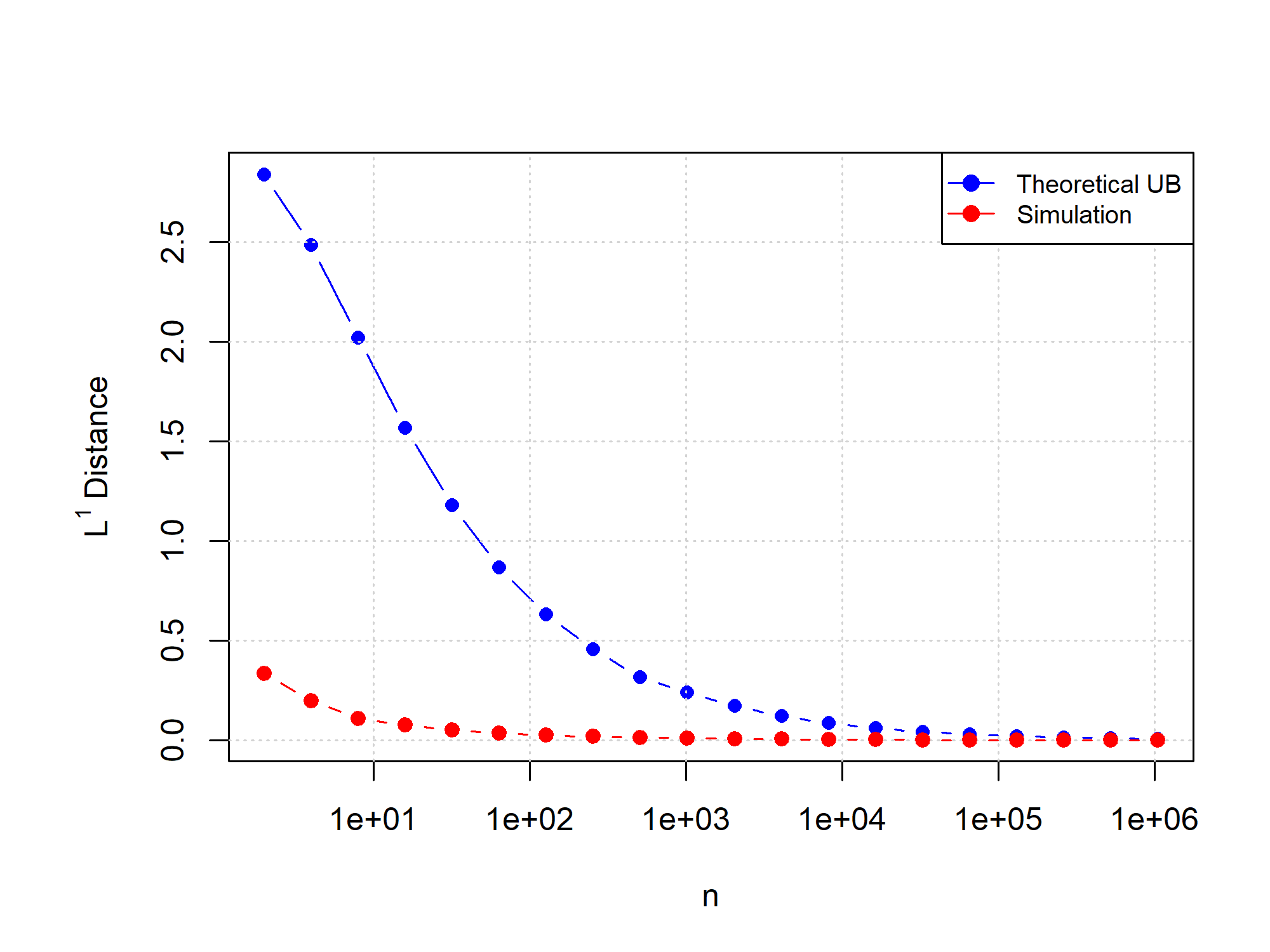}
	\caption{Simulated and theoretical upper bounds as functions of $n$, plotted on a logarithmic scale with a consistent $x$-axis spacing.}
	\label{fig:insure2}
\end{figure}

%%%%%%%%%%%%%%%%%%%%%%%%%%

%%%%%%%%%%%%%%%%%%%%%%%%%%%%%%%%%%%%%%%%%%%%%

\subsection{Generative AI} 

In this subsection, we consider an illustrative model for aggregate response times in a generative AI service. We use four groups to represent progressively longer delays, described qualitatively as minimal, moderate, substantial, and severe. These categories provide a simple way to distinguish several degrees of delay while keeping the model manageable; four is an illustrative choice rather than an empirically established number of operating states. Users are assigned to groups according to their allocation indices, with increasing mean response times across groups. As the number of users increases, additional users enter groups with longer expected response times, while the response-time distributions of users in earlier groups remain unchanged. This construction illustrates the application of our normal approximation results to a random sum with non-identically distributed summands; it is not intended as an empirically calibrated model of ChatGPT or a detailed queueing model.

\subsubsection{Theoretical results}

Let $n_1,n_2,n_3\in\mathbb{N}$ satisfy $n_1<n_2<n_3$. Let $X_1,X_2,\ldots$ be independent mean-zero random variables with $0<\Var(X_i)=\sigma_i^2<\infty$ for all $i\ge1$. Suppose that the variables are identically distributed within each of the four groups
\begin{equation*}
	\{X_1,\ldots,X_{n_1}\},\quad
	\{X_{n_1+1},\ldots,X_{n_2}\},\quad
	\{X_{n_2+1},\ldots,X_{n_3}\},\quad
	\{X_{n_3+1},X_{n_3+2},\ldots\},
\end{equation*}
with distributions that may differ across groups. Let $N\sim\mathrm{Poisson}(\lambda)$, where $\lambda>0$, be independent of the sequence. Define $W=\sum_{i=1}^{N}X_i$, with the empty sum interpreted as zero.

Let $n_0=0$ and $n_4=\infty$. Then
\begin{align}\label{TH:AI1}
	\sigma^2=\Var(W)
	&=\sum_{i=1}^{\infty}\sigma_i^2P(N\ge i)\nonumber\\
	&=\sum_{k=1}^{4}\sigma_{n_{k-1}+1}^2\sum_{i=n_{k-1}+1}^{n_k}P(N\ge i).
\end{align}

The following theorem gives a Wasserstein distance bound for the normal approximation of $W/\sigma$ in this four-group setting.

\begin{theorem}\label{Thm: ChatGPT} 
	Let $n_1,n_2,n_3\in\mathbb{N}$ satisfy $n_1<n_2<n_3$, and set $n_0=0$ and $n_4=\infty$. Let $X_1,X_2,\ldots$ be independent mean-zero random variables with $\Var(X_i)=\sigma_i^2>0$ and $\E|X_i|^3<\infty$ for all $i\ge1$. Suppose that the variables are identically distributed within each of the four groups $X_1,\ldots,X_{n_1}$, $X_{n_1+1},\ldots,X_{n_2}$, $X_{n_2+1},\ldots,X_{n_3}$, and $X_{n_3+1},X_{n_3+2},\ldots$, with distributions that may differ across groups. Let $N\sim\mathrm{Poisson}(\lambda)$, where $\lambda>0$, be independent of the sequence. Define $W=\sum_{i=1}^{N}X_i$, with the empty sum interpreted as zero, and let $\sigma^2=\Var(W)$ be given by \eqref{TH:AI1}. Then
	
	\begin{align}\label{Eqn:Chatgt}
		d_1\!\left(\mathcal{L}(W/\sigma),\mathcal{L}(Z)\right)
		&\le \frac{1}{\sigma^3}\sum_{k=1}^{4}
		\left(\E|X_{n_{k-1}+1}|^3+2\sigma_{n_{k-1}+1}^2\E|X_{n_{k-1}+1}|\right)
		\sum_{i=n_{k-1}+1}^{n_k}P^2(N\ge i)\nonumber\\
		&\quad+2\sqrt{\frac{2}{\pi}}
		\left(1-\frac{1}{\sigma^2}\sum_{k=1}^{4}\sigma_{n_{k-1}+1}^2
		\sum_{i=n_{k-1}+1}^{n_k}P^2(N\ge i)\right),
	\end{align}
	where $Z$ is a standard normal random variable. In particular,
	\begin{align}\label{Eqn:explicit Poisson}
		d_1\left(\mathcal{L}(W/\sigma),\mathcal{L}(Z)\right)
		&\le
		\frac{\max\limits_{1\le k\le4}\{\E|X_{n_{k-1}+1}|^3\}+2\max\limits_{1\le k\le4}\{\sigma_{n_{k-1}+1}^2\E|X_{n_{k-1}+1}|\}}{\min\limits_{1\le k\le4}\{\sigma_{n_{k-1}+1}^3\}\sqrt{\lambda}}
		\nonumber\\
		&\quad+
		\frac{2\max\limits_{1\le k\le4}\{\sigma_{n_{k-1}+1}^2\}}{\min\limits_{1\le k\le4}\{\sigma_{n_{k-1}+1}^2\}}
		\frac{1}{\sqrt{\pi\lambda}}.
	\end{align}
	
	If, in addition, $X_1,X_2,\ldots$ are identically distributed, then \eqref{Eqn:explicit Poisson} reduces to
	\begin{equation}\label{Eqn:Poisson}
		d_1\left(\mathcal{L}(W/\sigma),\mathcal{L}(Z)\right)
		\le
		\frac{\E|X_1|^3+2\sigma_1^2\E|X_1|}{\sigma_1^3\sqrt{\lambda}}
		+\frac{2}{\sqrt{\pi\lambda}}.
	\end{equation}
\end{theorem}
\proof

Let $I$ be defined as in \eqref{Idef}. By the independence of $I$ and $N$,
\begin{equation}\label{TH:AI2}
	P(N\ge I)=\frac{1}{\sigma^2}\sum_{k=1}^{4}\sigma_{n_{k-1}+1}^2\sum_{i=n_{k-1}+1}^{n_k}P^2(N\ge i).
\end{equation}
Applying Corollary~\ref{coro2.5} and using \eqref{TH:AI1} and \eqref{TH:AI2} yields \eqref{Eqn:Chatgt}.

By \eqref{TH:AI1},
\begin{equation*}
	\sigma^2\ge\min\limits_{1\le k\le4}\{\sigma_{n_{k-1}+1}^2\}\sum_{i=1}^{\infty}P(N\ge i)=\lambda\min\limits_{1\le k\le4}\{\sigma_{n_{k-1}+1}^2\}.
\end{equation*}
and
\begin{equation*}
	\sum_{i=1}^{\infty}P^2(N\ge i)\le\sum_{i=1}^{\infty}P(N\ge i)=\lambda,
\end{equation*}
this gives the first term in \eqref{Eqn:explicit Poisson} by bounding each group-specific moment contribution by its maximum over the four groups.

Applying the Poisson case of Lemma~\ref{cor:specialcase} gives the second term in \eqref{Eqn:explicit Poisson}. Finally, when all summands are identically distributed, \eqref{Eqn:explicit Poisson} reduces to \eqref{Eqn:Poisson}.
\bbox

\subsubsection{Simulation} 

Consider a random batch of users submitting queries of comparable complexity within a short observation window, which we treat as simultaneous for modeling purposes. We model the number of users by \(N\sim\mathrm{Poisson}(\lambda)\), where \(\lambda\) is the expected batch size. This assumption is motivated by a homogeneous Poisson arrival process over the observation window, with $\lambda$ equal to the arrival rate multiplied by the window length.

Users are indexed in the order in which their requests are allocated for processing, and we assume that this allocation determines their response-time group. The positive integers \(n_1<n_2<n_3\) are cumulative user-count thresholds defining four response-time groups. The first \(n_1\) users belong to the first group, the next \(n_2-n_1\) to the second, the next \(n_3-n_2\) to the third, and any remaining users to the fourth. Only indices up to the realized value of \(N\) are included. These groups specify response-time distributions, not successive service periods: users in a later group need not wait for all users in an earlier group to finish.

Let \(X'_i\) denote the time, in seconds, from submission of user \(i\)'s query until completion of the response. We assume that these response times are mutually independent and independent of \(N\), with an exponential distribution within each group. The exponential distribution provides a simple nonnegative, right-skewed model with a single scale parameter. Its memoryless property and the independence assumptions are simplifying modeling choices, rather than empirically established properties of generative AI response times.

We use four groups to illustrate successive increases in mean response time. This choice is illustrative and does not represent four empirically identified operating states of a particular service.

Set $n_0=0$ and $n_4=\infty$. For each $j\in\{1,2,3,4\}$, assume that
\begin{equation*}
	X'_i\sim\mathrm{Exp}(m_j),\qquad n_{j-1}<i\le n_j,
\end{equation*}
where $\mathrm{Exp}(m)$ denotes the exponential distribution with mean $m$. We assume that $0<m_1<m_2<m_3<m_4$, so that the mean response times increase across the four groups.

For each index $i$ in group $j$, define $X_i=X'_i-m_j$. The aggregate centered response time is
\begin{equation*}
	W=\sum_{i=1}^{N}X_i.
\end{equation*}
Thus, $W$ measures the deviation of the aggregate response time from its conditional expectation given $N$.

For $X'_i\sim\mathrm{Exp}(m_j)$, the first three absolute central moments are given by \eqref{absolute central moments}, with the exponential mean replaced by $m_j$. Substituting these moments into the first bound in Theorem~\ref{Thm: ChatGPT} yields the theoretical upper bounds reported below.

For the simulations, we set
\begin{equation*}
	m_1=0.2,\quad m_2=1,\quad m_3=5,\quad m_4=20,
\end{equation*}
where $m_j$ is the mean response time in seconds for group $j$. We choose the cumulative user-count thresholds $n_1=25{,}000$, $n_2=50{,}000$, and $n_3=100{,}000$. These parameter values were suggested by ChatGPT (OpenAI, January 15, 2026) and are used solely for numerical illustration; they are not empirically calibrated estimates of the response times or capacity thresholds of any particular service.

We consider
\begin{equation*}
	\lambda\in\{2^k:1\le k\le19\}\cup\{19384,22384,25384,40168,47568,54968,85536,105536,125536\}.
\end{equation*}
Table~\ref{tab:AI} compares the theoretical upper bounds with simulation-based estimates of $d_1(\mathcal{L}(W/\sigma),\mathcal{L}(Z))$. In addition to the powers of two, we include values of $\lambda$ around the three user-count thresholds to examine how the bounds and estimated distances behave as the expected number of users passes through these thresholds.

\begin{table}[htbp]
	\centering
	\begin{tabular}{|c|c|c|c|c|c|}
		\hline
		$\lambda$ & Upper bounds & Simulated distance &  $\lambda$ & Upper bounds & Simulated distance \\
		\hline
		2      & 2.303442184 & 0.211796991 & 19384  & 0.034265310& 0.002632326 \\
		4      & 1.846643308 & 0.158557773 & 22384  & 0.031893858& 0.002511350
		\\
		8      & 1.417833461 & 0.114162092 & 25384  & 0.140498287& 0.004407382 \\
		16     & 1.059220795 & 0.080924424 & 32768  & 0.055818578& 0.003344486 \\
		32     & 0.777429162 & 0.057079029 & 40168  & 0.039996452& 0.002554959 \\
		64     & 0.563964349 & 0.040386101 & 47568  & 0.032653545& 0.002321277 \\
		128    & 0.405901600 & 0.028594952 & 54968  & 0.077168816& 0.003802871 \\
		256    & 0.290573182 & 0.020198373 & 65536  & 0.042298414& 0.002611052 \\
		512    & 0.207243911 & 0.014301147 & 85536  & 0.026961063& 0.002003061 \\
		1024   & 0.147431924 & 0.010144281 & 105536 & 0.062300434& 0.002748798 \\
		2048   & 0.104694096 & 0.007218810 & 125536 & 0.031792132& 0.002210619 \\
		4096   & 0.074251823 & 0.005129138 & 131072 & 0.028769748& 0.001954450 \\
		8192   & 0.052614898 & 0.003643280 & 262144 & 0.012148304& 0.001521695 \\
		16384  & 0.037259807 & 0.002747920  & 524288 & 0.007406913& 0.001393969
		\\
		\hline
	\end{tabular}
	\caption{The comparison of theoretical upper bounds and simulated distances for different values of $\lambda$. The simulated distance is computed from 1,000,000 repeated simulations.}
	\label{tab:AI}
\end{table}

Figure~\ref{fig:AI} shows the simulation results, with a zoomed-in view around the three thresholds in Figure~\ref{fig:AI2}. The estimated distances exhibit local peaks near $n_1$, $n_2$, and $n_3$. As $\lambda$ passes through each threshold, the probability that $N$ includes summands from the next group increases. These newly included summands have larger variances and third absolute moments, which can temporarily reduce the accuracy of the normal approximation. The bound in \eqref{Eqn:Chatgt} accounts for this transition through the group-specific moments and the weights $P(N\ge i)$ and $P^2(N\ge i)$. Beyond each transition region, the estimated distances resume their overall decrease.

\begin{figure}[H]
	\centering
	\begin{subfigure}[t]{0.48\textwidth}
		\centering
		\includegraphics[width=\linewidth]{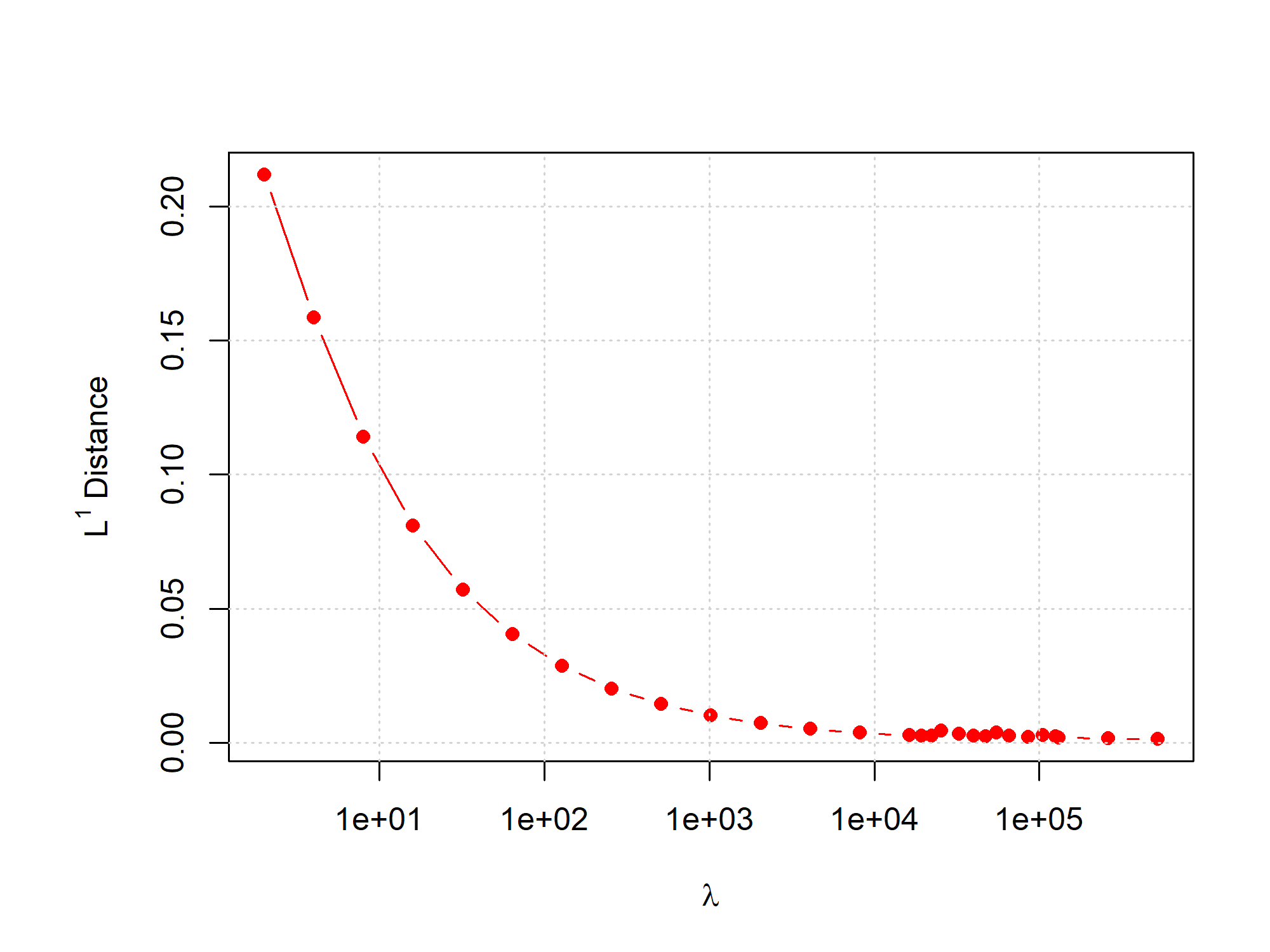}
		\caption{Simulated distances over the full range of $\lambda$.}
		\label{fig:AI}
	\end{subfigure}
	\hfill
	\begin{subfigure}[t]{0.48\textwidth}
		\centering
		\includegraphics[width=\linewidth]{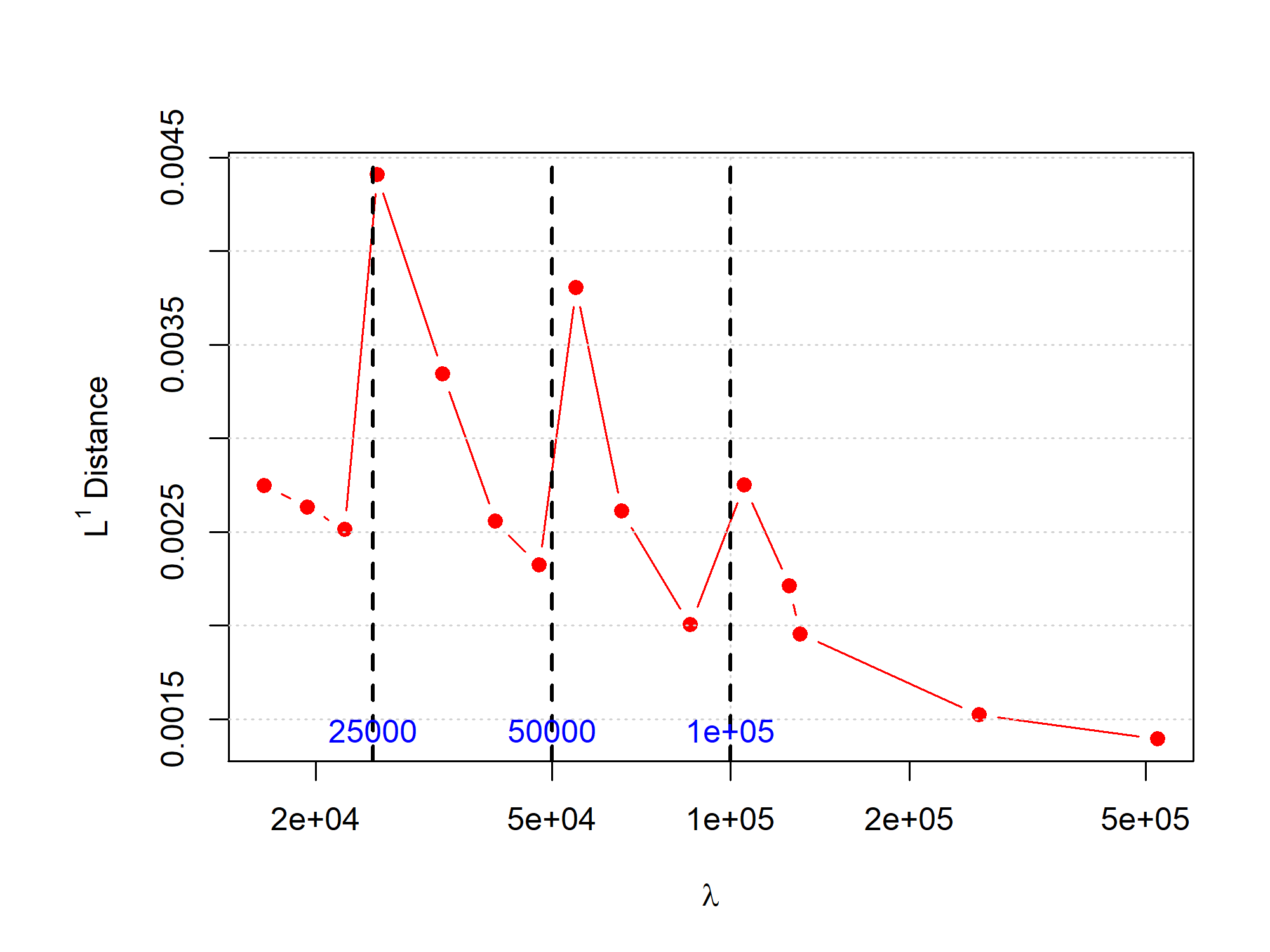}
		\caption{Zoomed-in view around the thresholds $n_1=25{,}000$, $n_2=50{,}000$, and $n_3=100{,}000$, indicated by dashed vertical lines.}
		\label{fig:AI2}
	\end{subfigure}
	\caption{Simulation-based estimates of $d_1(\mathcal{L}(W/\sigma),\mathcal{L}(Z))$ as functions of $\lambda$, with $\lambda$ plotted on a logarithmic scale.}
	\label{fig:AIcombined}
\end{figure}

Figure~\ref{fig:AI3} compares the simulated distances with the theoretical upper bounds. Both exhibit local increases near the three thresholds, suggesting that the bounds capture the qualitative changes in the accuracy of the normal approximation as additional groups begin to contribute to the random sum.

\begin{figure}[H]  % 'h' means here, but you can use 't' for top, 'b' for bottom, etc.
	\centering
	\includegraphics[width=0.7\textwidth]{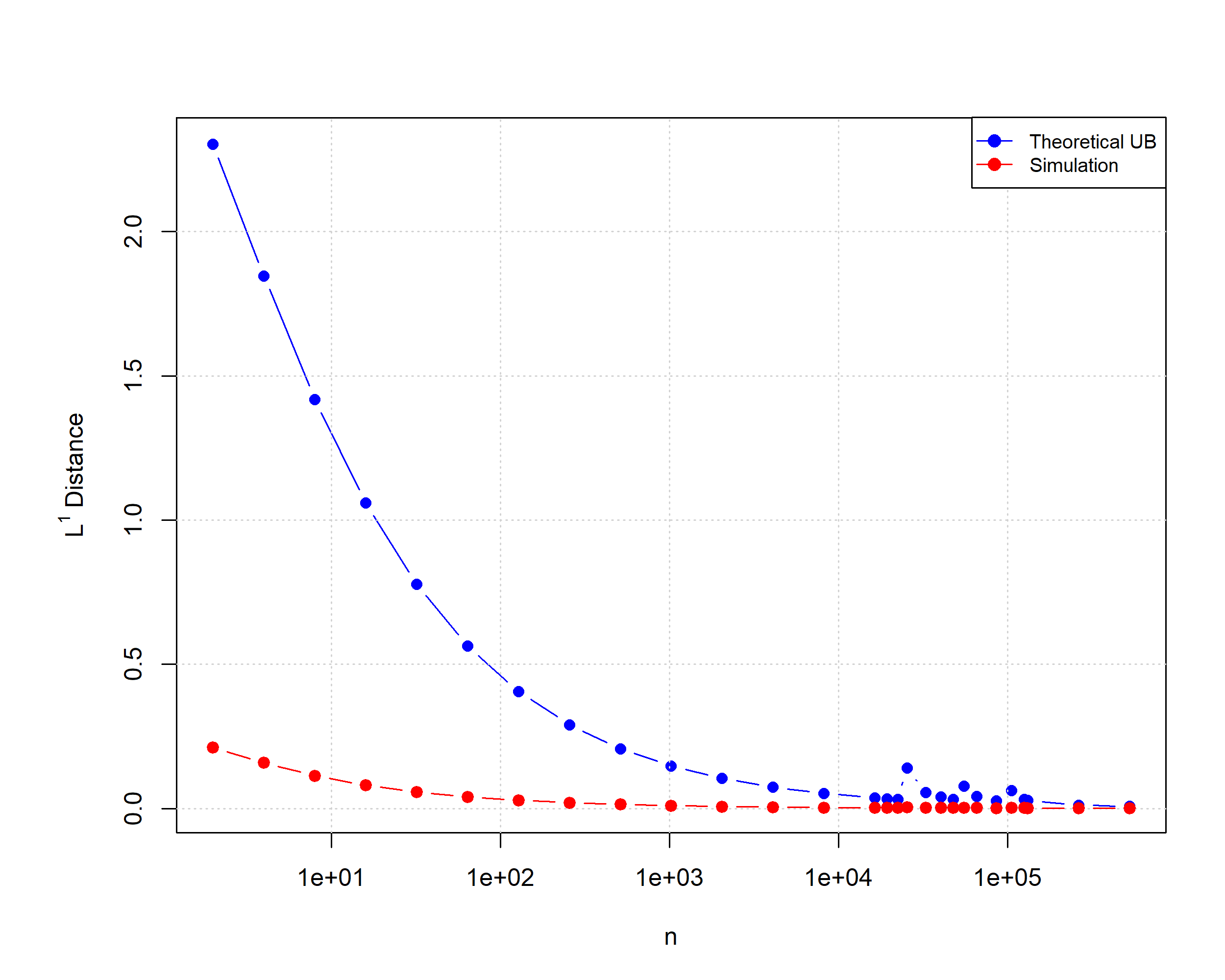}
	\caption{Simulated and theoretical upper bounds as functions of $\lambda$, plotted on a logarithmic scale with a consistent $x$-axis spacing.}
	\label{fig:AI3}
\end{figure}

\subsection{Comparison with existing bounds in the i.i.d. case} \label{sec:compare}

We compare the i.i.d.\ specializations of our results with existing bounds for binomial, Poisson, and hypergeometric random sums. Throughout this section, let $X_1,X_2,\ldots$ be i.i.d.\ mean-zero random variables with $\Var(X_1)=\sigma_1^2>0$ and $\E|X_1|^3<\infty$. Let $N$ be independent of $\{X_i\}_{i\geq1}$, and define $W=\sum_{i=1}^{N}X_i$ and $\sigma^2=\sigma_1^2\E[N]$. Let $Z$ be a standard normal random variable. Our bounds have the same convergence orders as the bounds considered below, although their constants can be larger. The underlying results also apply to independent, non-identically distributed summands.

\subsubsection{Binomial random sums}

Suppose that $N\sim\mathrm{Bin}(n,p)$, where $0<p<1$. Our bounds are given in Theorem~\ref{thm:BN}. For comparison, \citet{Dal22} obtained
\begin{equation}\label{Daly's Theoretical UB}
	d_1\left(\mathcal{L}(W/\sigma),\mathcal{L}(Z)\right)
	\le \frac{1}{\sigma_1^3\sqrt{np}}
	\left(\E|X_1|^3+2p\sigma_1^2\E|X_1|\right).
\end{equation}

For fixed $p$, both this bound and our explicit bound are of order $O(n^{-1/2})$. Daly's bound has a smaller coefficient than our explicit bound. However, this does not imply that it is smaller than our first bound for every $n$, since the latter retains the tail probabilities $P(N\geq i)$.

Table~\ref{tab:daly_comparison} illustrates this comparison using the centered Gamma variables and binomial parameter from the insurance application. It compares our first bound in Theorem~\ref{thm:BN} with \eqref{Daly's Theoretical UB}. Our bound is smaller at $n=2$, whereas Daly's bound is smaller for all other values of $n$ considered. Both bounds decrease at a rate consistent with $n^{-1/2}$.

\begin{table}[htbp]
	\centering
	\begin{tabular}{|c|c|c|c|c|c|}
		\hline
		\multirow{2}{*}{$n$}
		& Upper bounds
		& \multirow{2}{*}{Daly's upper bounds}
		& \multirow{2}{*}{$n$}
		& Upper bounds
		& \multirow{2}{*}{Daly's upper bounds}
		\\
		& from \eqref{Eqn: Binomial UB}
		&
		&
		& from \eqref{Eqn: Binomial UB}
		&
		\\
		\hline
		$2^1$  & 3.2140 & 3.9315 & $2^{11}$ & 0.1919 & 0.1229 \\
		$2^2$  & 2.8350 & 2.7800 &$2^{12}$ & 0.1364 & 0.0869 \\
		$2^3$  & 2.3179 & 1.9658 &$2^{13}$ & 0.0968 & 0.0614 \\
		$2^4$  & 1.8025 & 1.3900 &$2^{14}$ & 0.0686 & 0.0434 \\
		$2^5$  & 1.3579 & 0.9829 &$2^{15}$ & 0.0486 & 0.0307 \\
		$2^6$  & 1.0021 & 0.6950 &$2^{16}$ & 0.0344 & 0.0217 \\
		$2^7$  & 0.7296 & 0.4914 &$2^{17}$ & 0.0243 & 0.0154 \\
		$2^8$  & 0.5264 & 0.3475 &$2^{18}$ & 0.0172 & 0.0109 \\
		$2^9$  & 0.3775 & 0.2457 &$2^{19}$ & 0.0122 & 0.0077 \\
		$2^{10}$ & 0.2696 & 0.1738 &$2^{20}$ & 0.0086 & 0.0054 \\
		\hline
	\end{tabular}
	\caption{Numerical comparison of the first upper bound in Theorem~\ref{thm:BN} and Daly's bound in \eqref{Daly's Theoretical UB}, with $X_i=Y_i-\E[Y_i]$, where $Y_i$ has a Gamma distribution with shape $k=0.8203$ and scale $\theta=10962.0161$, and $N\sim\mathrm{Bin}(n,2583/8753)$.}
	\label{tab:daly_comparison}
\end{table}

\subsubsection{Poisson random sums}

Suppose that $N\sim\mathrm{Poisson}(\lambda)$, where $\lambda>0$. Our i.i.d.\ bound is given in \eqref{Eqn:Poisson} from Theorem \ref{Thm: ChatGPT}. In comparison, \citet{Dal22} obtained
\begin{equation}\label{Eqn:Daly Poisson}
	d_1\left(\mathcal{L}(W/\sigma),\mathcal{L}(Z)\right)
	\le \frac{\E|X_1|^3}{\sigma_1^3\sqrt{\lambda}}.
\end{equation}

Both bounds are of order $O(\lambda^{-1/2})$. Our bound contains additional terms and consequently has a larger coefficient.

\subsubsection{Hypergeometric random sums}

Suppose that $N$ has a hypergeometric distribution with population size $M$, $K$ successes, and sample size $n$, where $M\geq2$ and $1\leq K,n\leq M$. Our bounds follow from Theorem~\ref{SRM:Theoretical result} by setting $N_0=K$, and $m=M-K$, for $K<M$.

Let $N^s$ have the size-biased distribution of $N$. Applying the bound of \citet{Dal22} with a coupling satisfying $N^s\leq N+1$ almost surely, as in \citet{Dob15}, gives
\begin{equation}\label{Eqn:Daly hypergeometric}
	\begin{aligned}
		d_1\left(\mathcal{L}(W/\sigma),\mathcal{L}(Z)\right)
		&\leq \frac{2\sqrt{M}}{\sigma_1\sqrt{nK}}
		\left[
		\frac{\E|X_1|^3}{2\sigma_1^2}
		+\left(1-\frac{(M-K)(M-n)}{M(M-1)}\right)\E|X_1|
		\right].
	\end{aligned}
\end{equation}

For a fixed summand distribution, as $n\to\infty$ along a sequence of admissible population parameters with $K/M$ bounded away from zero, both this bound and our explicit bound are of order $O(n^{-1/2})$.

\section{Conclusion}
\label{sec:sum}

We develop an approximate zero-bias transformation and use Stein's method to derive upper bounds for the Wasserstein distance between a standardized random sum and the standard normal distribution. Our results allow independent summands with different distributions. For binomial, Poisson, and hypergeometric counts, the bounds yield explicit convergence rates under the assumptions specified in the corresponding results. In the i.i.d.\ case, these rates agree with those of the existing bounds considered in our comparisons, although our explicit bounds have larger constants.

We illustrate the results through models motivated by sampling from populations containing outliers, auto insurance payments, and generative AI response times. The insurance extension and the generative AI example demonstrate how the bounds accommodate groups of summands with different distributions. The numerical results show decreasing theoretical bounds and generally decreasing estimated distances as the expected number of summands increases, with local increases near the group thresholds in the generative AI model. For large expected counts, the estimated distances are subject to the numerical limitations discussed earlier.

These examples illustrate the applicability of our results to random sums with heterogeneous summands. Improving the constants while retaining this flexibility is a natural direction for future research.

\backmatter

\bmhead{Acknowledgements}

The authors sincerely appreciate the anonymous reviewer's fruitful and detailed comments. Their valuable guidance has led to a much-improved and better-organized manuscript. The authors would like to also thank Chutiphan Charoensuk, a research assistant supported by the Grants for Sci Seeding, Faculty of Science, Chulalongkorn University, for assistance with data preparation and numerical simulations. Nathakhun would like to specifically thank Dawud Thongtha for a discussion during the initial stage of this project.

%Acknowledgements are not compulsory. Where included they should be brief. Grant or contribution numbers may be acknowledged.

%Please refer to Journal-level guidance for any specific requirements.

\begin{comment}
	\bigskip\noindent

	\bigskip\noindent
	{\bf Data Availability} The artificial datasets used in this study can be obtained from the corresponding author upon request. The PEA dataset is confidential.
	
	\bigskip\noindent
	{\bf Ethics approval} The project did not involve ant studies with human or animal participants.
	
	\bigskip\noindent
	{\bf Conflict of interest} The authors declare no competing interest.

	\begin{comment}

		\bigskip
		\begin{flushleft}%
			Editorial Policies for:
			
			\bigskip\noindent
			Springer journals and proceedings: \url{https://www.springer.com/gp/editorial-policies}
			
			\bigskip\noindent
			Nature Portfolio journals: \url{https://www.nature.com/nature-research/editorial-policies}
			
			\bigskip\noindent
			\textit{Scientific Reports}: \url{https://www.nature.com/srep/journal-policies/editorial-policies}
			
			\bigskip\noindent
			BMC journals: \url{https://www.biomedcentral.com/getpublished/editorial-policies}
		\end{flushleft}
	\end{comment}

	\section*{Declarations}
	
	\noindent\textbf{Declaration of generative AI in the writing process:}
	During the preparation of this manuscript, the authors used ChatGPT (OpenAI) solely for language editing, grammar checking, and improving the clarity of the manuscript. The authors reviewed and approved all AI-assisted text and take full responsibility for the final content of the manuscript.
	
	\noindent\textbf{Funding:} Not applicable.
	
	\noindent\textbf{Conflict of interest/Competing interests:} The author declares no conflict of interest.
	
	\noindent\textbf{Data availability:} Data are available upon request.

	%\begin{appendices}

	%\end{appendices}
	
	
	\bibliography{sn-bibliography}% common bib file

@book{CGS11,
	author    = {Chen, Louis H. Y. and Goldstein, Larry and Shao, Qi-Man},
	title     = {Normal Approximation by {Stein}'s Method},
	publisher = {Springer},
	address   = {New York},
	year      = {2011}
}

@article{Dal22,
	author  = {Daly, Fraser},
	title   = {Gamma, {Gaussian} and {Poisson} approximations for random sums using size-biased and generalized zero-biased couplings},
	journal = {Scandinavian Actuarial Journal},
	volume  = {6},
	pages   = {471--487},
	year    = {2022}
	
}

@misc{Xiaomeng,
	author    = {Sun, Xiaomeng},
	title     = {Car Insurance Claim Data},
	year      = {2018},
	publisher = {Kaggle},
	url       = {https://www.kaggle.com/datasets/xiaomengsun/car-insurance-claim-data}
}

@article{EKJ09,
	author  = {El Karoui, Noureddine and Jiao, Yiqing},
	title   = {Stein’s method and zero bias transformation for {CDO} tranche pricing},
	journal = {Finance and Stochastics},
	volume  = {13},
	pages   = {151--180},
	year    = {2009}
}

@incollection{GR05,
	author    = {Goldstein, Larry and Reinert, Gesine},
	title     = {Zero biasing in one and higher dimensions, and applications},
	booktitle = {Stein's Method and Applications},
	series    = {Lecture Notes Series, Institute for Mathematical Sciences, National University of Singapore},
	volume    = {5},
	pages     = {1--18},
	publisher = {Singapore University Press},
	address   = {Singapore},
	year      = {2005}
}

@article{GR97,
	author  = {Goldstein, Larry and Reinert, Gesine},
	title   = {Stein’s method and the zero bias transformation with application to simple random sampling},
	journal = {Annals of Applied Probability},
	volume  = {7},
	pages   = {935--952},
	year    = {1997}
}

@article{dobriban2025statistical,
	title={Statistical methods in generative {AI}},
	author={Dobriban, Edgar},
	journal={arXiv preprint arXiv:2509.07054},
	year={2025}
}

@book{klugman2012loss,
	author    = {Klugman, Stuart A. and Panjer, Harry H. and Willmot, Gordon E.},
	title     = {Loss Models: From Data to Decisions},
	edition   = {4},
	year      = {2012},
	publisher = {John Wiley \& Sons},
	address   = {Hoboken, NJ}
}

@book{denuit2007actuarial,
	title={Actuarial modelling of claim counts: {Risk} classification, credibility and bonus-malus systems},
	author={Denuit, Michel and Mar{\'e}chal, Xavier and Pitrebois, Sandra and Walhin, Jean-Fran{\c{c}}ois},
	year={2007},
	publisher={John Wiley \& Sons},
	address = {Chichester}
}

@book{barnett1994outliers,
	author    = {Barnett, Vic and Lewis, Toby},
	title     = {Outliers in Statistical Data},
	edition   = {3},
	year      = {1994},
	publisher = {John Wiley \& Sons},
	address   = {Chichester}
}

@article{robbins1948asymptotic,
	title={The asymptotic distribution of the sum of a random number of random variables},
	author={Robbins, Herbert},
	journal={Bulletin of the American Mathematical Society},
	volume={54},
	number={12},
	pages={1151--1161},
	year={1948}
}

@article{renyi1963central,
	title={On the central limit theorem for the sum of a random number of independent random variables},
	author={R{\'e}nyi, A},
	journal={Acta Mathematica Academiae Scientiarum Hungarica},
	volume={11},
	number={1},
	pages={97--102},
	year={1963},
	publisher={Springer}
}

@inproceedings{chaidee2008berry,
	title={{Berry--Esseen} bound for independent random sum via {Stein}’s method},
	author={Chaidee, N and Neammanee, K},
	booktitle={Int. Math. Forum},
	volume={3},
	number={15},
	pages={721--738},
	year={2008}
}

@book{gnedenko2020random,
	title={Random summation: Limit theorems and applications},
	author={Gnedenko, Boris V and Korolev, Victor Yu},
	year={1996},
	publisher={CRC press},
	address   = {Boca Raton, FL}
}

@article{sunklodas2015normal,
	title={On the normal approximation of a negative binomial random sum},
	author={Sunklodas, Jonas Kazys},
	journal={Lithuanian Mathematical Journal},
	volume={55},
	number={1},
	pages={150--158},
	year={2015},
	publisher={Springer}
}

@article{barbour2006normal,
	title={Normal approximation for random sums},
	author={Barbour, Andrew D and Xia, Aihua},
	journal={Advances in applied probability},
	volume={38},
	number={3},
	pages={693--728},
	year={2006},
	publisher={Cambridge University Press}
}

@article{NW2017,
	title={Stein's method using approximate zero bias couplings with applications to combinatorial central limit theorems under the Ewens distribution},
	author={Wiroonsri, Nathakhun},
	journal={ALEA},
	volume={14},
	pages={917--946},
	year={2017}
}

@inproceedings{stein72,
	title={A bound for the error in the normal approximation to the distribution of a sum of dependent random variables},
	author={Stein, Charles},
	booktitle={Proceedings of the sixth Berkeley symposium on mathematical statistics and probability, volume 2: Probability theory},
	volume={6},
	pages={583--603},
	year={1972},
	organization={University of California Press}
}

@article{Ross11,
	author  = {Nathan Ross},
	title   = {Fundamentals of {S}tein's Method},
	journal = {Probability Surveys},
	volume  = {8},
	pages   = {210--293},
	year    = {2011},
	doi     = {10.1214/11-PS182},
	mrnumber = {2861132},
	issn    = {1549-5787}
}

@article{YonghintJantai2025,
	author  = {Yonghint, Nat and Jantai, Wasamon},
	title   = {Poisson {A}pproximation for {S}top-Loss {M}etrics of {O}rder 1 and 2},
	journal = {Sankhya A},
	volume  = {87},
	pages   = {302--326},
	year    = {2025},
	doi     = {10.1007/s13171-025-00399-5}
}

@article{JW25,
	title={On {C}ombinatorial {C}entral {L}imit {T}heorems with {D}ifferent {U}nderlying {P}ermutations {V}ia {A}pproximate {Z}ero {B}iasing},
	author={Jantai, Wasamon and Wiroonsri, Nathakhun},
	journal={Sankhya A},
	volume={87},
	number={2},
	pages={261--301},
	year={2025},
	publisher={Springer}
}

@article{Dob15,
	author  = {D{\"o}bler, Christian},
	title   = {New {Berry--Esseen} and {Wasserstein} bounds in the {CLT}
	for non-randomly centered random sums by probabilistic methods},
	journal = {ALEA, Latin American Journal of Probability and Mathematical Statistics},
	year    = {2015},
	volume  = {12},
	number  = {2},
	pages   = {863--902}
}
	
\end{document}